\documentclass[a4paper,12pt,article]{iopart}
\usepackage{currfile}
\usepackage{graphicx}
\usepackage{amssymb}
\usepackage{amstext}
\usepackage[expansion=true]{microtype}

\usepackage{bm}
\usepackage{epstopdf}
\usepackage{color}
\usepackage{hyperref}
\usepackage{mhequ}
\usepackage{times}
\usepackage{amsthm}
\usepackage{accents}
\usepackage{enumitem}
\usepackage[titletoc,title]{appendix}
\usepackage{perpage} 
\MakePerPage{footnote} 
\usepackage{booktabs}

\usepackage{sub_JP}

\def\Label#1{}

\def\complex{{\bf C}}
\let\phi=\varphi
\let\kappa=\varkappa

\usepackage{chngcntr}
\counterwithin*{equation}{section}
\renewcommand{\i}{{\mathrm{i}}}

\renewcommand{\theequation}{\thesection.\arabic{equation}}
\renewcommand{\thesection}{\arabic{section}}
\renewcommand{\thesubsection}{\thesection.\arabic{subsection}}

\def\dt{\frac{\d }{\d t}}
\def\d{{\rm d}}

\let\epsilon=\varepsilon
\let\theta=\vartheta

\let\rho=\varrho
\def\IM{\bI\cdot\bmm}
\def\mphi{\bmm\cdot\bphi}
\def\IP{(\bI,\bphi)}
\def\NNR{^{(0,\NR)}}

\def\ONR{^{(1,\NR)}}

\def\x#1{^{(#1)}}
\def\xx#1{^{[#1]}}

\def\ie{{\it{i.e.}}}
\def\eg{{\it{e.g.}}}

\def\bmm{{\bm {\mu}}}

\def\bphi{\bm{\phi}}

\def\bI{{\bf{I}}}

\def\NN{{\cal N}}

\def\BB{{\cal B}}
\def\DD{{\cal D}}
\def\DDR{{\cal D^{\real}\kern-0.8em}}

\def\OOs{{\cal O}}
\def\OOd{\widehat{\cal O}}

\def\torus{{\bf T}}
\def\RR{{R}}
\def\integer{{\bf Z}}
\def\real{{\bf R}}
\def\lref#1{Lemma~\ref{#1}}
\def\aref#1{Assumption~\ref{#1}}
\def\dref#1{Definition~\ref{#1}}
\def\fref#1{Fig.~\ref{#1}}
\def\cref#1{Corollary~\ref{#1}}
\def\pref#1{Proposition~\ref{#1}}
\def\tref#1{Theorem~\ref{#1}}

\def\sref#1{Sect.~\ref{#1}}
\def\rref#1{Remark~\ref{#1}}

\def\halffact{{\textstyle{\frac{1}{2!}}}}
\def\thirdfact{{\textstyle{\frac{1}{3!}}}}

\def\norm#1#2{\|#1\|_{#2}}
\def\bnorm#1#2{\left\|#1\right\|_{#2}}
\newtheorem{theorem}{Theorem}[section]
\newtheorem{proposition}[theorem]{Proposition}
\newtheorem{lemma}[theorem]{Lemma}
\newtheorem{observation}[theorem]{Observation}
\newtheorem{corollary}[theorem]{Corollary}

\newtheorem{assumption}[theorem]{Assumption}
\newtheorem{definition}[theorem]{Definition}

\newtheorem{remark}[theorem]{Remark}

\definecolor{mydarkgreen}{rgb}{0.0, 0.5, 0.0}

\def\poiss#1#2{\left\{#1,#2\right\}}
\def\ad{{\rm ad}}

\def\rs{{r,\sigma}}
\def\Im{{\rm Im}}

\def\R{{\rm R}}
\def\NR{{\rm NR}}
\def\S{{\rm S}}

\def\supp{{\rm{supp}}}
\def\fin{{\mathrm {fin}}}
\let\tilde=\widetilde

\begin{document}
\null
\title[Dissipative Chains of Rotators]{Energy Dissipation in Hamiltonian Chains of Rotators}
\author{No\'e Cuneo$^1$,  Jean-Pierre Eckmann$^{2,3}$, C. Eugene Wayne$^4$}
\address{$^1$ Department of Mathematics and  Statistics, McGill
  University, Montreal, Canada}
\address{$^{2}$ D\'epartement de Physique Th\'eorique, University of
  Geneva, Switzerland}
\address{$^{3}$ Section de Math\'ematiques, University of
  Geneva, Switzerland}
\address{$^4$ Department of Mathematics and Statistics,
Boston University, USA}
\begin{abstract}
We discuss, in the context of energy flow in high-dimensional systems and
Kolmogorov-Arnol'd-Moser (KAM) theory, the behavior of a chain of rotators (rotors)
which is purely Hamiltonian, apart from dissipation at just one end.
We derive bounds on the dissipation rate which become arbitrarily small in certain physical regimes,
and we present numerical evidence that these bounds are sharp.
We relate this to the decoupling of non-resonant terms as is
known in KAM problems.
\end{abstract}

\tableofcontents

\section{Introduction}\Label{s:intro}

 In this paper, we consider a chain of coupled rotators.  Our basic system is
Hamiltonian,
and we are specifically interested in the way in which
energy is transported through the system.  We study this by adding
dissipation
to the first rotator,
and then considering initial conditions in which the energy of the system is
localized in a rotator far
from the one which dissipates energy.  
Our main result is a lower bound on the energy dissipation rate. The lower bound becomes arbitrarily small as either the energy of the fast rotator is made very large, or as the number of rotators increases. Although at the moment our rigorous estimates are ``one-sided'', we present numerical evidence that they are sharp, in the sense that in some regimes energy dissipation does scale like our lower bound, indicating that energy transport in the system is indeed extremely slow.

Hamiltonian systems, non-equilibrium problems and the interplay of rigorous
results with computer
based insight have been central themes in ``Nonlinearity'' since its
origin 30 years ago
and as such we feel
this paper is an appropriate way to celebrate this anniversary.
We thus use this introduction to sketch a (relatively personal) view on some
highlights in the subject of our paper, with emphasis on papers which have
appeared in Nonlinearity.

We start with Hamiltonian systems:  Here, our problem is connected with
ideas
like the KAM and
Nekhoroshev theorems which have appeared repeatedly in Nonlinearity
\cite{Benettin_Gallavotti_1986,Nekoroshev_1977,Loshak_1992,Lochak_Neistad_1992,Lochak_1993,Poeschel_1993,MacKay_Aubry_1994}.
The common feature of
both of these approaches is the  realization that in certain limiting
regimes
(usually small coupling) Hamiltonian systems are close to their decoupled
variants.  One typically takes advantage of this by transforming the
system into some sort of normal form by canonical transformations. In the
particular problem considered
in this paper the situation is more reminiscent of the Nekhoroshev than KAM
theorems since we will
also make a finite number of transformations of our original system to
bring it to a form in which
the dissipation can be directly estimated.
Another key idea is that periodic orbits play an especially important role in the
stability
of systems with respect to perturbations.
Early papers in this context
are \cite{Nekoroshev_1977, Loshak_1992, Poeschel_1993}.
In particular, Lochak pointed out that
the evolution of Hamiltonian systems in the
neighborhood of periodic orbits occurred only on a very long time scale.  Of
particular interest
in our context, are the spatially localized periodic orbits known as
``breathers'' \cite{MacKay_Aubry_1994}.  One respect in which our
problem
differs from this
preceding work is that  we do not deal with small coupling, but with high
frequencies. It seems this problem is not readily transformed into one with
small coupling, but is rather close to ideas having to do with breathers,
isolated high-energy states in extended systems
\cite{Rey-Bellet_Thomas_2000,Rey-Bellet_Thomas_2002,Hairer_2009,Hairer_Mattingly_2009,Cuneo_Eckmann_Poquet_2015,Cuneo_Eckmann_2016,Cuneo_Poquet_2016}.

We find that our system approaches a
``quasi-stationary'' state in which the rotators oscillate with
an amplitude that decreases in a regular way from
the rotator in which the original energy was concentrated, to the
dissipating rotator.  These states
are reminiscent of breather solutions in Hamiltonian lattices \cite{MacKay_Aubry_1994,Bambusi_1996,Bambusi_Nekhoroshev_1998,Bambusi_1999} which are known to be stable over long
time intervals even in the purely Hamiltonian setting.
That these breather states should inhibit transport in coupled lattice
systems
has already been
suggested by several authors \cite{Rey-Bellet_Thomas_2000,Rey-Bellet_Thomas_2002,Hairer_2009,Hairer_Mattingly_2009}.

They are also reminiscent of metastable states in weakly viscous fluid flows
where invariant structures present in the conservative limit (\ie,~Euler
equations) are replaced by families of solutions which seem to determine the
long relaxation toward the rest state when weak dissipation is introduced
\cite{Beck_Wayne_2013,Gallay_2012}.  In a similar fashion, in our problem,
the breather
solutions which are exact, periodic solutions in the conservative limit seem
to be replaced by nearly periodic dissipative solutions along which the
system evolves toward the rest state.

The combination of such problems with dissipation is more recent: A careful
study was done by \cite{Rey-Bellet_Thomas_2002,DeRoeck_Huveneers_2015}. These studies
have in
common that they estimate the rate of dissipation.
Our goal in the present work is more restricted than these prior
studies.  We
are particularly interested
in identifying the states of the system which give rise to the slowest
possible
dissipation rate.  Because
we consider problems in which the initial energy is localized far from the
center of dissipation, these
correspond to states in which the transport of energy through the
lattice is especially slow.
In another direction dissipation in Hamiltonian systems has become an
intensive
object of study in planetary studies \cite{CALLEJA2013978,Celletti_Chierchia_2009,Celletti_Lhotka_2012,Celletti_Lhotka_2013}.

In \cite{Celletti_Lhotka_2012,Celletti_Lhotka_2013} the authors study the evolution of trajectories in
nearly integrable Hamiltonian systems subjected to very general, small
dissipative forces.  They show that solutions remain close to the integrable
solutions for long times by using a combination of canonical
transformations, and then representing the dissipation in the canonically
transformed variables.

In another, related, direction are
studies of out-of-equilibrium systems, usually in a stochastic context (the
stochasticity modeling heat-baths)
\cite{Eckmann_Pillet_Rey-Bellet_1999a,
Eckmann_Pillet_Rey-Bellet_1999b,
Rey-Bellet_Thomas_2000,
Rey-Bellet_Thomas_2002,Gallavotti_Cohen_2004,Eckmann_Young_2006,Eckmann_Hairer_2000}. In most of these papers, one
of the problems is to show that transport actually happens, and that
the energy of the system does not blow up due to the energy input from
the heat baths \cite{Eckmann_Pillet_Rey-Bellet_1999b,Hairer_2009}.
There is now a quite abundant literature about non-equilibrium steady
states of
simple, 1-dimensional systems, where masses or rotators are coupled to heat
baths \cite{Kipnis_Marchioro_Presutti_1982,Bonetto_Lebowitz_Rey-Bellet_2000,Bernardin_Olla_2011}. One of the issues in this
context
is to prove the existence of stationary states for such systems. It has been
recognized early on that the convergence for ``arbitrary'' initial
conditions to
such stationary states can be very slow, in particular, when one of the
masses
or rotators has very high energy \cite{Lepri_Livi_Politi_2003,Lepri_Livi_Politi_2016}.
This phenomenon is rightly attributed to a loss of effective
coupling between a high energy site and its low energy neighbors,
because of an
averaging effect of the rapidly oscillating force.
Combining some of the two aspects above
  \cite{DeRoeck_Huveneers_2015} studied heat
transport in a chain of oscillators, in the limit of weak coupling (see also
\cite{Lefevere_Schenkel_2004}).

In this paper, we address a much simpler problem, namely a purely
deterministic system (\ie, without any stochastic heat baths) with only
dissipation and an otherwise purely Hamiltonian
interaction. Obviously, such a system will eventually converge to the
state where nothing moves, but we are interested in the initial phase
 where the energy is still very high and the system
loses energy (very slowly) through the dissipation. The advantage of
considering such a simple system is the possibility of understanding
it quite completely.

More precisely, we study in detail the behavior of a
chain of $n$ coupled rotators (with \emph{nearest neighbor coupling})
\emph{with localized dissipation on the first of these rotators}.
We are interested in the evolution of this system in a region of phase
space where almost all of the energy of the system is located
in one of the sites.
 Such a setup is of course
reminiscent of breathers, but now we have
dissipation.

 In particular, our results are consistent with the suggestion put
forth by \cite{Lepri_Livi_Politi_2003} that even the notion of
stationary state in large (infinite)
such systems might need a new definition. Because certain regions of phase
space (\eg, the states we study in which the initial energy is localized far
from any source of dissipation) take so long to converge to the stationary
state, the actual limiting state may not be the most relevant quantity for
practical effects like heat conduction or the Fourier law.

\section{Model and statement of results}

We fix the number $n>1$ of rotators once and for all, and use the phase space $\Omega \equiv \real^n\times
\torus^n$ for the actions $\bI=(I_1,\dots,I_n)$ and the angles
$\bphi=(\phi_1,\dots,\phi_n)$. We adopt the convention $\torus = \real /2\pi$. We shall often write $x = \IP$. We consider Hamiltonians of the form
\begin{equ}\label{e:nn}
  H\x0\IP = \frac{1}{2}\sum_{i=1}^{n} I_i^2 + \sum_{i=1}^{n-1} U_{i}(\phi_{i+1}-\phi_{i})= h\x0 (\bI)+f\x0(\bphi)~,
\end{equ}
where the functions $U_i : \torus \to \real $ are real, smooth interaction potentials subject to
the following two assumptions.
\begin{assumption}\label{d:22}The Fourier series of each $U_i$ contains finitely many modes.
\end{assumption}
\begin{assumption}\label{d:23}Each $U_i$ is {\em
    non-degenerate} in the sense that there is no $\phi\in \torus$ for
  which $U_{i}'(\phi) = U_{i}''(\phi) = 0$. In other words, the interaction force
and its derivative never vanish simultaneously.
\end{assumption}
Without loss of generality, we also assume that the $U_i$ contain no constant (and physically irrelevant) Fourier mode.

We fix a dissipation coefficient $\gamma>0$ and consider, for $i=1, \dots, n$, the evolution equations
\begin{equa}[e:motion2]
  \dot I_i &= -\partial_{\phi_i} H\x0 -\delta_{i,1} \gamma I_1~,\\
  \dot \phi_i &= \partial_{I_i} H\x0~.
\end{equa}


We fix $2\leq k \leq n$.
We will concentrate on domains in phase space where $I_k$ is of order $L \gg 1$, and the other actions are of order $\rho \sim 1$. Without loss of generality, we will consider only the case $L>0$, since the case $L<0$ is identical by symmetry.

\begin{definition}
For $L,\rho>0$, we denote by $\BB_{L,\rho}$ the open set
\begin{equs}
\BB_{L, \rho} = \left\{(\bI, \bphi) \in \Omega: | I_i - L \delta_{i,k}| < \rho,~ i=1, 2, \dots, n\right\}~.
\end{equs}
\end{definition}

Our main result is

\begin{theorem}\label{t:main}
Consider the system of equations  \eref{e:motion2}, under Assumptions~\ref{d:22} and \ref{d:23}. Fix $\alpha > 0$ and $\rho>0$. Then, there exist constants $L_0, \rho^*, C_1>0$ such that for all
$L\geq L_0$ and all initial conditions $x(0) \in \BB_{L,\rho}$,
the following holds for $T\equiv \alpha L^{2k-3}$:
\begin{enumerate}
\item[(i)]	We have
\begin{equs}
x(t) \in \BB_{L, \rho^*}~, \quad 0\leq t \leq T~.
\end{equs}

\item[(ii)]	Moreover,
  \begin{equ}\label{e:dissipation}
    H(x(T))-H(x(0)) \leq   -C_1 \frac{T}{L^{4k-6}}~.
  \end{equ}
\end{enumerate}
\end{theorem}

\begin{remark}
Part (ii) says that the amount of energy loss per unit of time
is {\em at least} proportional to $L^{6-4k}$, when averaged over
times of order $L^{2k-3}$.
Numerical investigation (see \sref{sec:illustration}) indicates that
the power $4k-6$ in the theorem is optimal.
\end{remark}

\begin{remark}
\aref{d:23} is crucial for the theorem above:
as we will show in \sref{sec:nondeg}, dropping \aref{d:23}
may lead to decay rates that are much slower than \eref{e:dissipation}.
\end{remark}

\subsection{The relation to KAM}

Part of our study is closely related to the usual techniques of KAM
theory. However, the addition of dissipation to an otherwise purely
Hamiltonian problem requires significant modifications of the
classical KAM machinery.

  In general, KAM-like calculations are done for weak coupling. In our
  case, the coupling is of order 1 but the energies are very high. It
  is tempting to envisage a scaling which connects our problem to the
  well-studied KAM techniques. We have not been able to find such a
  scaling, probably due to the fact that rotators are akin to pinning
  potentials of infinite power, \ie, potentials like $x^{2m}$ with
  $m=\infty$ (see \cite[Remark 1.1]{Cuneo_Eckmann_Poquet_2015}).

\section{Phenomenology}\label{sec:illustration}

In the regime of interest, the $k$'th rotator holds most of the energy of the system,
and tends to effectively decouple from its neighbors (or its neighbor if $k = n$),
as the interaction forces involving the $k$'th  site oscillate very rapidly and
average out.

The
idea of the proof of \tref{t:main} is as follows. In \sref{sec:general} and \sref{sec:k1},
we will construct new canonical coordinates $\tilde x = (\tilde \bI, \tilde \bphi)$, which
are very close to $x$ when $L$ is large, and which make the decoupling explicit.
More precisely, the coordinates $\tilde x$ obey  some dynamics where the site $k$ is decoupled from the rest of the system, except
for some ``remainders'' of magnitude $L^{2-2k} \ll 1$ (see \eref{eq:HtildeOR}).
Part (i) will follow from this and technical considerations (see \pref{p:wasGronwall}).

Moreover, we will obtain (see \rref{rem:optimalpowers}) the scaling
\begin{equa}[eq:scalingTildeImI]
	\tilde I_i - I_i & \sim L^{1-2|k-i|} \qquad (i\neq k)~,\\
	\tilde I_k - I_k & \sim L^{-1}~,\\
	\tilde \phi_i - \phi_i & \sim L^{-2|k-i|} \qquad (i\neq k)~,\\
	\tilde \phi_k - \phi_k & \sim L^{-2}~.
\end{equa}
In particular, we have
\begin{equs}
I_1 = \tilde I_1 + P_1~,
\end{equs}
where $P_1$ scales like $L^{3-2k}$. To prove Part (ii) of \tref{t:main},
we will use the decomposition
\begin{equs}[eq:dtHintro]
\dt H = - \gamma I_1^2 = -\gamma \tilde I_1^2 - \gamma P_1^2 -2\gamma \tilde I_1 P_1~.
\end{equs}
The key role will be played by the second term in the right-hand side (see \sref{sec:dissipnotzero}),
which is of order $L^{6-4k}$ and will give rise to the dissipation
rate in \eref{e:dissipation}. The third term has no sign, and we will
show in \sref{sec:partii} that it has negligible effect thanks to its oscillatory nature.
It is the comparison of the second and third contributions in \eref{eq:dtHintro}
that leads to the choice $T \sim  L^{2k-3}$ (see \rref{rem:choiceT}).

In addition to being the central ingredient of our proof of \tref{t:main},
the construction of the new coordinates $\tilde x$ leads to some very
interesting observations, which we illustrate here without proof. 

First, we observe numerically (see \sref{ssec:numerics}) that the system quickly reaches a
quasi-stationary state, where the actions $I_i$ oscillate with frequency of order $L$
and small amplitude around some fixed value $\langle I_k \rangle \approx L$ for $I_k$, 
and around zero for $I_i$, $i\neq k$. At the same time, $\phi_k$ rotates rapidly, 
while the angles $\phi_i$, $i\neq k$ oscillate with small amplitude around an equilibrium position
of the interaction potentials.
More precisely, the amplitude of the oscillations of the actions appears to scale like
\begin{equs}[eq:scalingImI]
I_i &\sim   L^{1-2|k-i|} \qquad (i\neq k)~,\\
	I_k - \langle I_k\rangle  & \sim L^{-1}~.
\end{equs}

This can be understood as follows. By construction, the action $\tilde I_k$ is
approximately constant, which with \eref{eq:scalingTildeImI} explains 
\eref{eq:scalingImI} for site $k$.
To explain \eref{eq:scalingImI} for $i<k$, we observe that the system
$(\tilde \phi_1, \dots, \tilde \phi_{k-1}, \tilde I_1, \dots, \tilde I_{k-1})$
has, up to small corrections, the dynamics of a chain of $k-1$ rotators
with dissipation on the first one, and initial energy of order 1.
As there is no decoupling phenomenon among these $k-1$ ``slow''
rotators, this subsystem rapidly dissipates most of its energy and
approaches some rest position
where $\tilde I_i \approx 0$ for all $i< k$, and where all the $\tilde \phi_i$,
$i< k$, are close to a local equilibrium of the interaction potentials. This and \eref{eq:scalingTildeImI}
explain \eref{eq:scalingImI} for $i<k$.

Explaining why the same holds for $i>k$ (in case $k<n$) is less obvious. Since there is no direct dissipation on the right of rotator $k$, such a symmetric decay of the amplitude of oscillation might seem unexpected. First, the symmetry in \eref{eq:scalingTildeImI} follows from the construction of the variables $\widetilde \phi_i$ and $\widetilde I_i$, and does not rely on dissipation. The fact that \eref{eq:scalingTildeImI} implies \eref{eq:scalingImI} also for $i>k$ can be understood as follows.
While the variables $(\tilde I_k, \tilde \phi_k)$ are decoupled from the rest of
the system up to terms of order $L^{2-2k}$, the two subsystems
$(\tilde \phi_1, \dots, \tilde \phi_{k-1}, \tilde I_1, \dots, \tilde I_{k-1})$
and
$(\tilde \phi_{k+1}, \dots, \tilde \phi_{n}, \tilde I_{k+1}, \dots, \tilde I_{n})$
still interact through some terms of order $L^{-2}$ (see \rref{rem:leftrightcoupled}),
which is much larger than $L^{2-2k}$ as soon as $k\geq 3$. This seems to allow
the energy of the subsystem $(\tilde \phi_{k+1}, \dots, \tilde \phi_{n}, \tilde I_{k+1}, \dots, \tilde I_{n})$
to dissipate rather rapidly (\ie, much faster than any significant change in the energy
of rotator $k$ can be observed) so that the argument above also yields \eref{eq:scalingImI} 
for $i>k$.

Coming back to \eref{e:dissipation}, the following can be said. During
the initial transient phase, the instantaneous dissipation rate can be
much larger than $L^{6-4k}$. However, once the quasi-stationary state is
reached, then by \eref{eq:scalingImI} we really have
$\dt H = - \gamma I_1^2  \sim -L^{6-4k}$,
indicating that the power of $L$ in \eref{e:dissipation} is optimal
(in the sense that choosing initial conditions that are already in the
quasi-stationary state yields a dissipation rate which is no faster
than what we claim).

\subsection{Numerical illustration}\label{ssec:numerics}

We illustrate here the properties of the quasi-stationary state
for short chains, assuming that  $U_i(\phi_{i+1}-\phi_i) = -\cos(\phi_{i+1}-\phi_i)$ for all $i$.
We first consider $n=k=4$ and $\gamma=1$, with
$L=10$ and $L=100$. For all numerical simulations, we choose
the initial condition
$x_0 \in \BB_{L, \rho}$ such that $I_i(0) = 0$ for $i\neq k$, 
$I_k(0) = L$, and $\phi_i(0) = 0$ for all $i$.
In \fref{fig:fig0}, we depict the maximum of $|I_i|$
over intervals of length $2\pi /L$ (corresponding to periods of $\phi_4$).

\begin{figure}[htb]
\begin{center}
	\includegraphics[width=\textwidth, trim={0.3cm, 1cm, 0.35cm, 0}]{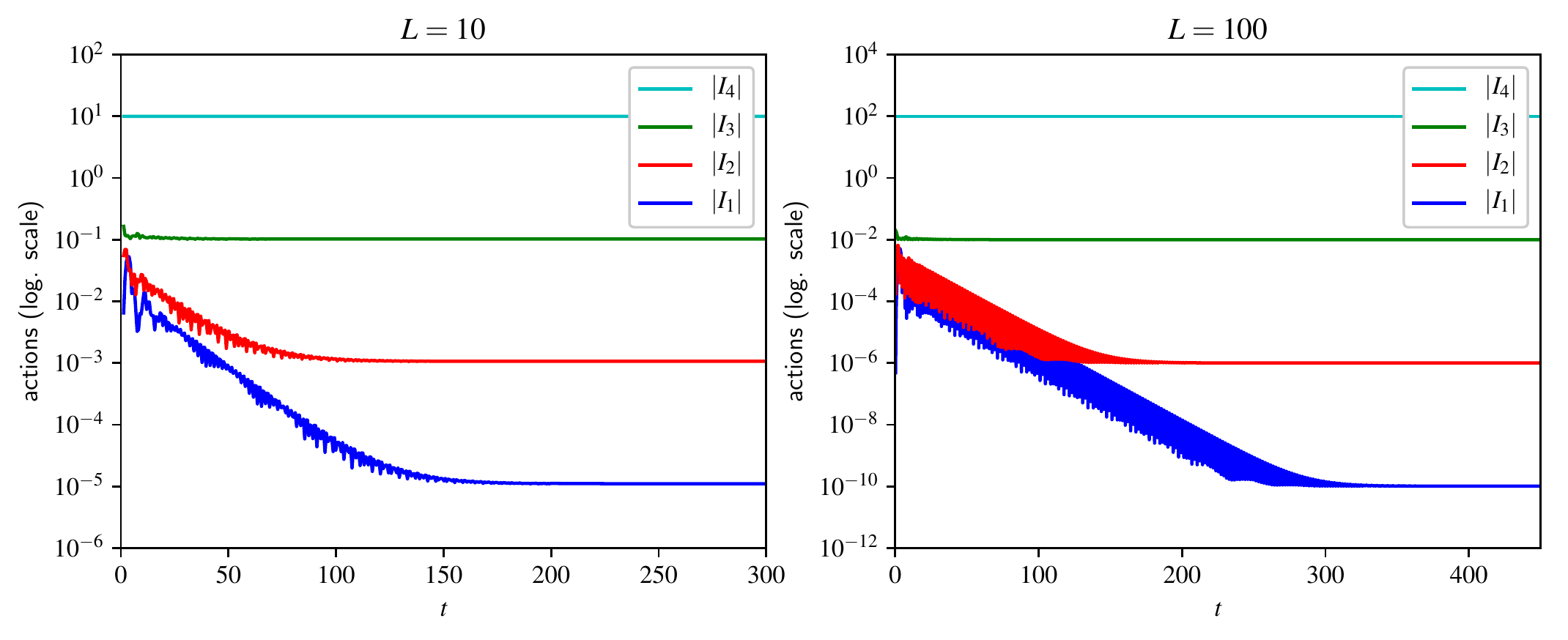}
\end{center}
  \caption{Maximum of $|I_i|$ over intervals of length $2\pi/L$ for $n=k=4$ and 
    $\gamma=1$, with $L=10$ (left) and $L=100$ (right).
}\label{fig:fig0}
\end{figure}

The transient phase and the quasi-stationary state are clearly visible on \fref{fig:fig0}.
Observe in particular that the  $I_i$
``equilibrate'' one after the other, with $I_1$ needing the longest
time.
Note that the vertical axis is $\log|I_i|$ and that the observed values
in the quasi-stationary state scale like $|I_{i}|\sim L^{1-2(n-i)}$, 
which is \eref{eq:scalingImI}.
The equidistance of the lines reflects the relation  $ \frac{|I_{i}|}{|I_{i+1}|}\sim \frac{1}{L^2}$.
The agreement here is perfect because for cosine potentials the numerical prefactors in
\eref{eq:scalingImI} happen to be 1 (see \sref{sec:asymptoticeqs}), so that the $I_i$, $i\neq k$,
are expected to oscillate with amplitude $(1+\OOs(L^{-1}))L^{1-2(n-i)}$.

\fref{fig:fig2} illustrates the phases of the (appropriately rescaled) oscillations 
in the quasi-stationary state in the same situation as above ($k=n=4$).
The observations corroborate again \eref{eq:scalingImI}.

\begin{figure}[htb]
\begin{center}
	\includegraphics[width=0.85\textwidth, trim={0, 1cm, 0, 0}]{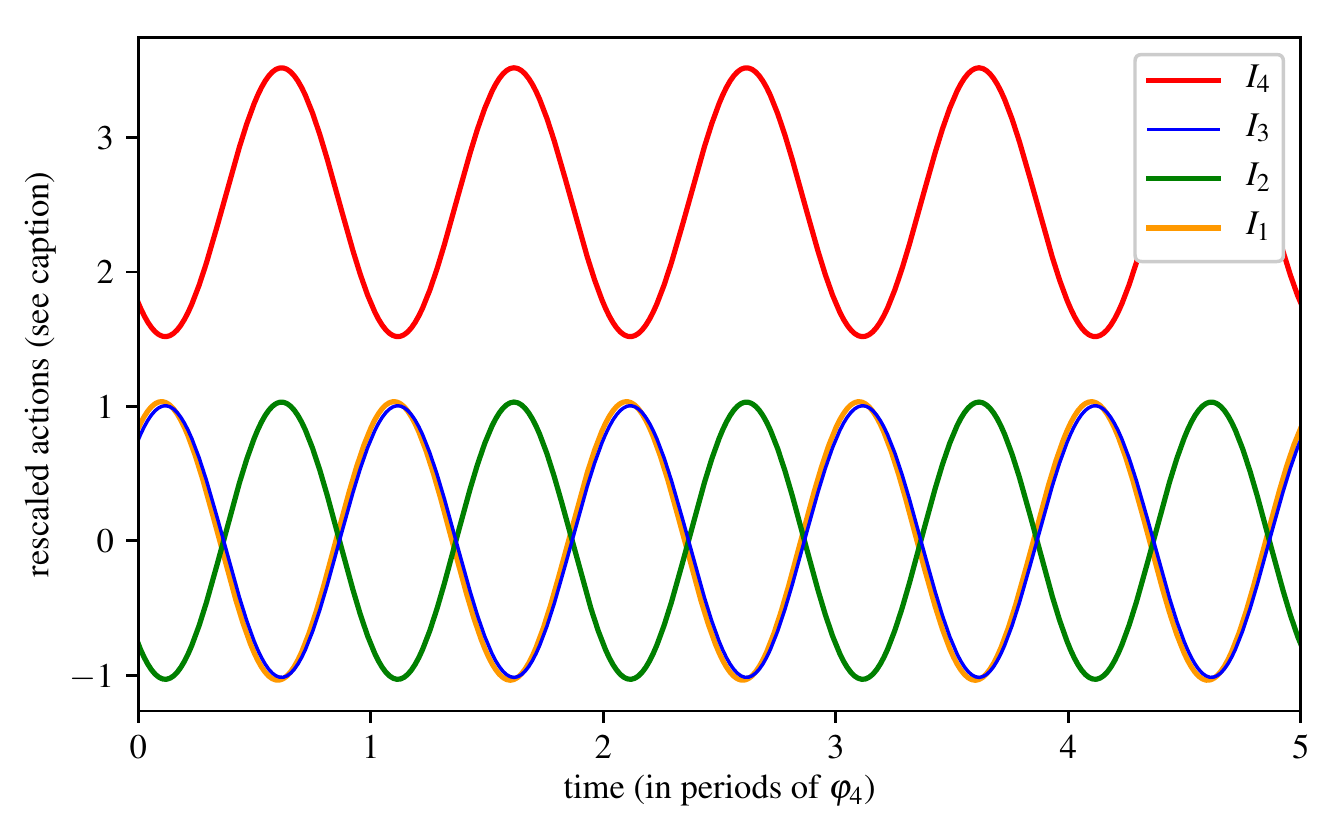}
\end{center}
  \caption{Oscillations in the quasi-stationary state $n=k=4$, $\gamma= 1$ and $L=10$.
 The amplitudes of $I_i$ are
    rescaled by the predicted factors $L^{2|4-i|-1}$, except for
  $I_4$ (the fast rotator), which is displayed as 
    $(I_4-\langle I_4 \rangle)\cdot L$ and shifted upward for better visibility.}
\label{fig:fig2}
\end{figure}

In \fref{fig:fig1}, we show the results for the case $n=6$ and
$k=4$. We observe that the amplitudes in the quasi-stationary
state depend
only on the absolute value of $i-k$, and are compatible with \eref{eq:scalingImI}.
\begin{figure}[htb]
  \begin{center}\includegraphics[width=0.85\textwidth, trim={0, 1cm, 0, 0}]{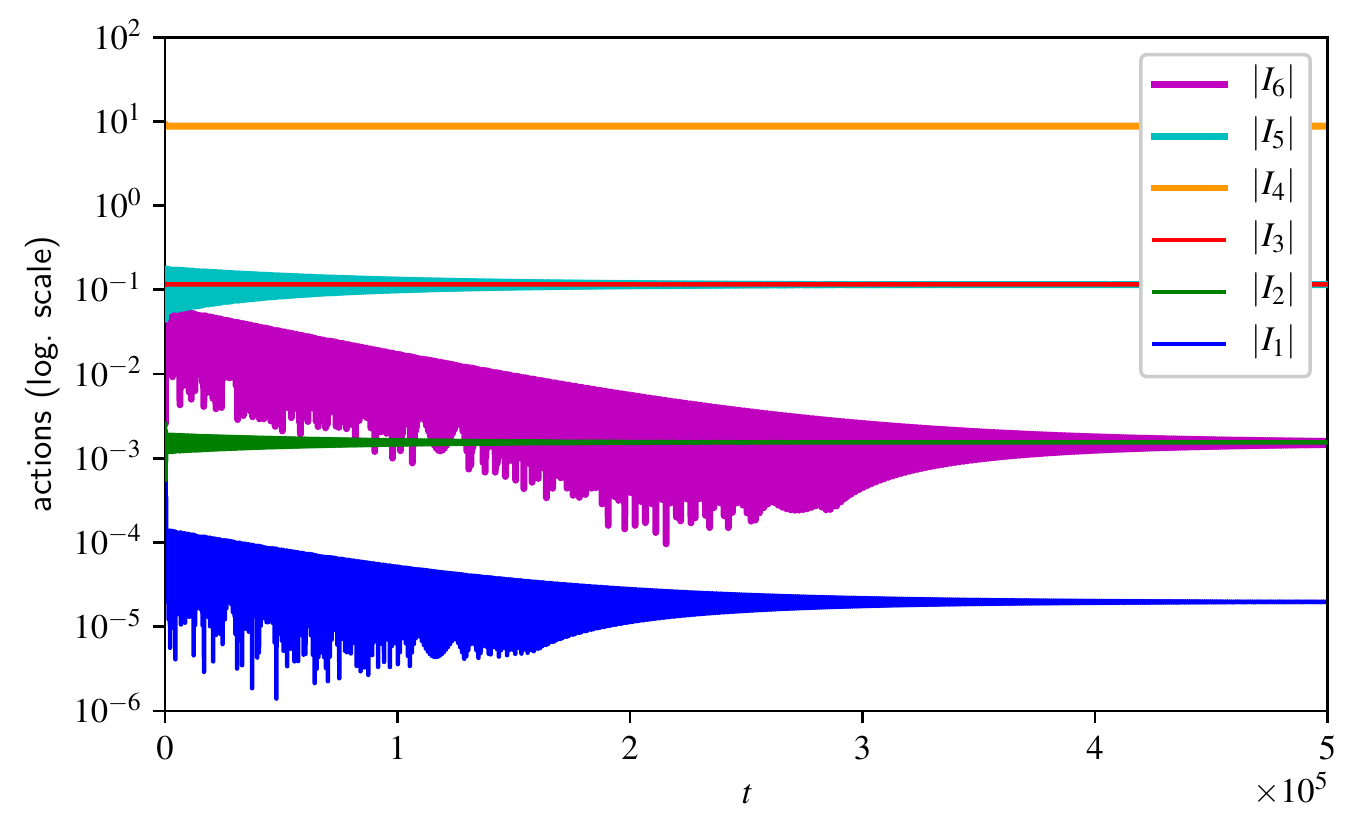}
\end{center}
  \caption{Maximum of $|I_i|$ over intervals of length $2\pi/L$ for $n=6$, $k=4$, $\gamma =2 $ and $L=9$.
}\label{fig:fig1}

\end{figure}

\section{Notation and tools}\label{sec:Notationstools}

We will consider analytic functions on a complex neighborhood of the phase space $\Omega$.

\begin{definition}\label{def:drs}
For any $L, r, \sigma > 0$, we define a complex domain
(in $\complex^{2n}$) by
\begin{equa}
  \DD_{L,r,\sigma} =\{
(\bI,\bphi) \in\complex^{2n}~: ~ |I_i-L\delta_{i,k}| < rL~, |\Im \phi_i| <\sigma, i=1, \dots, n\}~,
\end{equa}
and a norm on analytic, $\bphi$-periodic functions
\begin{equa}\label{e:normrs}
 \norm{f}{L,\rs}=\sup_{(\bI,\bphi) \in \DD_{L,\rs}} |f(\bI,\bphi)|~.
\end{equa}
We say that a $\bphi$-periodic function  $f$ is {\em admissible}
if
there exist $L_0, r, \sigma > 0$ such that for all $L\geq L_0$, $f$ is analytic on
$\DD_{L,r,\sigma}$ and $ \norm{f}{L,\rs}<\infty$.
\end{definition}

By assumption, $f\x 0$ is admissible
(we assume throughout that Assumptions~\ref{d:22} and \ref{d:23}
are in force). In fact, $ \norm{f\x 0}{L,\rs}$ is finite for every $r,\sigma$, and
independent of $L$.

\begin{remark}\label{rem:distsets}

We make the distinction between $\DD_{L, \rs}$ and $\BB_{L, \rho}$ clear:
\begin{itemize}
	\item The set $\DD_{L, r, \sigma}$ is complex, and the actions $I_i$, $i\neq k$ are allowed to grow with $L$. We will mostly work in terms of $\DD_{L, r, \sigma}$ when discussing analytic properties of  functions. In particular we need this scaling in order
to obtain good bounds on the derivatives with respect to $I_i$,
using Cauchy's theorem (see \lref{lem:poissonf1f2}).
\item The set $\BB_{L, \rho}$ is real, and the actions $I_i$, $i\neq k$ are absolutely bounded by $\rho$, independently of $L$. This will be more convenient when working with the orbits of the system, since as stated in \tref{t:main} (i), they remain for long times in such sets.
\end{itemize}

\end{remark}

\begin{remark}\label{rem:inclusionBBDD}
For all $\rho, r, \sigma > 0$, we have $\BB_{L, \rho} \subset \DD_{L,r,\sigma}$ as soon as $L \geq \rho/r$.\end{remark}

\subsection{Fourier series and resonance}

Given an admissible function $f$, we write
its Fourier series with respect to $\bphi$ as
\begin{equ}[eq:fgeneric]
 	f(\bI, \bphi)=\sum_{\bmm\in\NN}f_\bmm(\bI)e^{\i \mphi},
 \end{equ}
where $\NN \subset \integer^n$. The Fourier series above
is well-defined in $\DD_{L,r,\sigma}$ for $r, \sigma$ small enough and all $L$ large enough.

\begin{definition}\label{d:norms}Given a function $f$ and its
Fourier series \eref{eq:fgeneric}, we define
\begin{itemize}
	\item[] $\supp(\bmm) = \{i: \mu_i \neq 0\}$,
   \item[] $|\bmm| =\max\{|\mu_i|, i=1,\dots,n\}$,
    \item[]$|\NN| =\sup\{|\bmm|~: ~\bmm\in\NN\}$.
\end{itemize}
\end{definition}

\subsection{Sets of interactions and the non-resonance condition}

We will use the following notation:
\begin{definition}
We decompose any $\NN \subset \integer^n $ as a disjoint union $\NN=\NN^\NR\cup\NN^\R$:
\begin{equa}[e:splitting]
  \NN^{\NR}&=\big\{ \bmm \in \NN ~:~ \mu_k \ne 0\big\}~,
 \quad  \NN^{\R}=\big\{ \bmm \in \NN ~:~ \mu_k =  0\big\}~,
\end{equa}
where $\NR$ stands for {\em non-resonant} interactions, and
$\R$ stands for {\em resonant} interactions.
Given a function $f$ of the form \eref{eq:fgeneric},
we decompose correspondingly
\begin{equ}\label{e:RNR}
f =  f^\R+f^\NR~,
\end{equ}
where
\begin{equa}\label{e:RNR2}
f^{\R}=\sum_{\bmm\in\NN^\R}f_\bmm(\bI)
  e^{\i \mphi}~,
\quad f^{\NR}=\sum_{\bmm\in\NN^\NR}f_\bmm(\bI)
  e^{\i \mphi}~.
\end{equa}
\end{definition}

For example, we have the decomposition $f\x{0} = f\x{0, \R} + f\x{0, \NR}$, where
\begin{equs}[eq:decompF0]
f\x{0, \R} &= \sum_{i\in \{1, \dots, n-1\}\setminus \{k-1, k\}}U_i(\phi_{i+1}-\phi_i)~,\\
f\x{0, \NR} &=  \sum_{i \in \{1, \dots, n-1\}\cap\{k-1, k\} }U_i(\phi_{i+1}-\phi_i)~.
\end{equs}

Note that the non-resonant part of the interaction includes all terms in which any rotator couples to rotator $k$, and the resonant parts include any terms with no interaction with rotator $k$.

\begin{remark}
	This notion of {\em resonant} and {\em non-resonant} is justified as follows.
When $I_k$ is very large and the other actions are small, we have the following approximate dynamics for short times:
\begin{equs}
\phi_k(t) & \approx  \phi_k(0) + t I_k~,\\
\phi_i(t) &	\approx \phi_i(0)~.
\end{equs}
With respect to this dynamics, $e^{\i \bmm \cdot \bphi(t)}$ essentially oscillates with frequency $I_k \mu_k$ if $\mu_k \neq 0$, and is essentially constant if $\mu_k = 0$. Taking the average over one period of length $2\pi / I_k$, the non-resonant terms cancel out, while the resonant ones remain unchanged.
Thus, one can say that $f^\NR$ is the {\em oscillatory} part of $f$, while $f^\R$ is the {\em average} part of $f$.
\end{remark}

\begin{lemma}\label{lem:lhalf}  Let $\bmm \in \integer^n $ be non-resonant (\ie, $\mu_k \neq 0$),
and assume $r<{1}/({2n |\bmm|})$. Then
for any $\IP\in\DD_{L,\rs}$ one has
\begin{equ}\Label{e:NR}
  |\IM|> \frac{L}{2}~.
\end{equ}
\end{lemma}

\begin{proof} This is a consequence of the following easy computation:
  \begin{equs}
    |\IM| &\ge  |I_k \mu_k| - |\sum_{i\neq k} I_i \mu_i| \ge |\mu_k|L(1-r) - |\bmm| (n-1)rL\\
&\ge L(1-nr|\bmm|)~.
  \end{equs}
\end{proof}

Conversely, note that if $\bmm$ is resonant, $\bI \cdot \bmm$ may vanish
on any $\DD_{L,\rs}$. That distinction will be crucial in the sequel, and we will need to check
that all factors of $\bI \cdot \bmm$ appearing in the denominators of the functions we construct have $\bmm$ non-resonant.

\subsection{Lie series}\label{sec:lieseries}
We
will eliminate the non-resonant terms by a Lie
transformation, writing $H'=H \circ \Phi$, where $\Phi$ is the time-1 flow generated by
the Hamiltonian vector field $X_\chi$ of a well-chosen Hamiltonian function $\chi$.

To fix the notation,
we use here the terminology of P\"oschel \cite{Poeschel_1993}.
He defines the canonical transformation $ \Phi$ as
\begin{equ}[eq:defexpo]
G \circ \Phi = \sum_{s=0}^{\infty} \frac{1}{s!} \ad^s_{\chi} G\ ,
\end{equ}
with
\begin{equ}
\ad_{\chi}^0 G = G \ ,\ \  \ad_{\chi}^s G= \{ \ad_{\chi}^{s-1} G, \chi \}\ .
\end{equ}
The Poisson bracket $\{\cdot,\cdot\}$ is defined as
\begin{equ}\Label{e:poisson}
\{ F, G\} = \sum_{i=1}^n \left( \frac{ \partial F}{\partial \phi_i}\frac{ \partial G}{\partial I_i} -  \frac{ \partial F}{\partial I_i}\frac{ \partial G}{\partial \phi_i} \right)\ .
\end{equ}
With a slight abuse of notation, we will write
\begin{equs}
	G\circ  \Phi\equiv e^{\chi} G~,
\end{equs}
where the right-hand side really means $e^{\ad_{\chi}}G$ and is defined by \eref{eq:defexpo}.
The transformation $\Phi$ is canonical, and its inverse is given by $G \circ \Phi^{-1} = e^{-\chi} G$. 

The Hamiltonian evolution of some function $F$  with respect to the Hamiltonian $H$ is given by the Poisson bracket $\{ F, H\}$.
The evolution of the rotators in our model consists of two pieces: a Hamiltonian piece which we can express in terms of the Poisson bracket and the dissipative term that affects the first rotator. More explicitly, we have
\begin{equa}
\dt{I}_{i} = \{ I_{i}, H\} - \gamma I_1 \delta_{1,i}~, \quad   \dt{\phi}_{i} = \{ \phi_{i}, H\}\ .
\end{equa}

We will transform the Hamiltonian part of the equations with the aid of the Lie transform method to almost completely decouple the $k^{\mathrm {th}}$ rotator from the remainder of the system. We will then have to examine carefully the effects of these transforms on the dissipative term.

\subsection{Eliminating non-resonant interactions}\label{subs:eliminateNR}

We now explain in detail the main iteration of our process in an abstract setting.
Assume we have a Hamiltonian of the form
\begin{equs}
H = h\x0 + f + g= h\x0 + f^\NR + f^\R  + g~,
\end{equs}
where the decomposition \eref{eq:fgeneric} of $f$ contains {\em finitely many} terms,
and $g$ is any function that we cannot, or do not need to remove.
We would like to remove the non-resonant interactions  $f^\NR$
from the
new Hamiltonian by changing coordinates as in \sref{sec:lieseries},
with some well-chosen $\chi$.

The transformation $\Phi$ transforms $H$ into a new Hamiltonian $H'$ given by
\begin{equa}[eq:Hprime]
  H' &=e^{\chi }H =h\x0+f^\NR + f^\R +g +
\{h\x0 ,\chi \} + \{f + g ,\chi \} \\
& \qquad + \halffact
\{\{H ,\chi \},\chi \}+ \thirdfact
\{\{\{H ,\chi \},\chi \},\chi \}+\dots~\\
& = h\x0+f^\NR + \{h\x0 ,\chi \} + f^\R +g \\
& \qquad + \sum_{\ell=2}^\infty \frac 1 {\ell!}\ad_{\chi}^\ell h\x0 + \sum_{\ell=1}^\infty \frac 1 {\ell!}\ad_{\chi}^\ell (f + g)~.
\end{equa}
In order to remove the $f^\NR$ from \eref{eq:Hprime}, the idea is to choose $\chi$ so that
 \begin{equ}[eq:aeliminerNR]
 	\poiss{h\x0}{\chi}=-f^\NR~.
 \end{equ}
Such a $\chi$ is constructed as follows.
\begin{definition}
	For any function $f$ as in \eref{eq:fgeneric},
with $\NN$ finite, we let
\begin{equs}[eq:defQ]
Q f \equiv -\i\sum_{\bmm\in\NN^{\NR}}
  \frac{f_\bmm(\bI)}{\IM}e^{\i \,\mphi}~.
\end{equs}
\end{definition}
Since $\NN^\NR$ is finite, we have by \lref{lem:lhalf}
that the denominators $\bI \cdot \bmm$ appearing here are well defined and larger
than $L/2$ in absolute value on $\DD_{L,\rs}$ for some appropriate $r, \sigma$,
and all large enough $L$. Thus, $Qf$ is well defined on that domain.
Moreover, since these denominators scale like $L$,
we reasonably expect $Qf$ to be ``small'' when $L$ is large.
This will be made precise later. We observe also that if $f$ is real (\ie, $f_{-\bmm} = \bar f_{\bmm}$)
then so is $Qf$.
Finally, by construction,
\begin{equ}[eq:cancelNR]
	\poiss{h\x0}{Qf}=  -\sum_{i=1}^n I_i \partial_{\phi_i} Qf=- \sum_{\bmm\in\NN^{\NR}}f_\bmm(\bI)
  e^{\i \mphi} = - f^{\NR}~,
\end{equ}
so that choosing
\begin{equ}
\chi = Qf
\end{equ}
guarantees that \eref{eq:aeliminerNR} holds.
As a consequence, \eref{eq:Hprime} becomes
\begin{equa}[eq:iterationstep]
  H' &= h\x0+ f^\R +g + \sum_{\ell=2}^\infty \frac 1 {\ell!}\ad_{\chi}^\ell h\x0 + \sum_{\ell=1}^\infty \frac 1 {\ell!}\ad_{\chi}^\ell (f + g)\\
&= h\x0+ f^\R + g + \sum_{\ell=1}^\infty \frac 1 {\ell!} \ad_{\chi}^\ell\left( f + g- \frac {f^\NR}{\ell+1}\right)~.
\end{equa}

Thus, $f^\NR$ has indeed been removed from $H'$ (while $f^\R$  remains untouched).
The price to pay is the
appearance of some new interactions in the form of the infinite series above. The idea is then to eliminate part of this new term by another canonical transformation, and to iterate this procedure as many times as needed.

We address now the natural question of whether this infinite series converges, and whether it is actually smaller than the interaction $f^\NR$ that we have removed.

\subsection{Orders of magnitude}

\begin{definition}\label{def:OOd}
Let $f$ be admissible, and let  $s\in \real$. We say that
  \begin{equ}\label{e:oo}
    f=\OOd(L^{s})
  \end{equ}
if there exist constants $r, \sigma, L_0, C > 0$ such that
for all $L\geq L_0$,
\begin{equ}\label{e:oo2}
  \norm{f}{L,\rs} \le C L^s~.
\end{equ}
\end{definition}

The following properties are proved in Appendix~\ref{sec:thebounds}, and
will often be used without reference.
\begin{lemma}\label{lem:propsOO}
Let $f = \OOd(L^s)$ and $g = \OOd(L^z)$. Then,
\begin{enumerate}[label=(\alph*)]
	\item $f+g = \OOd(L^{\max(s,z)})$,
	\item $fg = \OOd(L^{s+z})$,
	\item $\partial_{\phi_i} f = \OOd(L^{s})$,
	\item $\partial_{I_i} f = \OOd(L^{s-1})$,
	\item $\{f, g\} = \OOd(L^{s+z-1})$,
	\item $\ad_g^\ell f = \OOd(L^{s+\ell(z-1)})$,
	\item  if $(a_\ell)_{\ell \geq 0}$ is a bounded sequence, then for all $0\leq \ell_0 \leq \ell_1 \leq \infty$,
\begin{equs}
	\sum_{\ell=\ell_0}^{\ell_1} \frac{a_\ell}{\ell!} \ad_{g}^\ell f = \OOd(L^{s+\ell_0(z-1)}),
\end{equs}
\item if $f$ has the form \eref{eq:fgeneric} and $\tilde \NN \subset \NN$, then
\begin{equ}
 	\sum_{\bmm\in\tilde \NN}f_\bmm(\bI)e^{\i \mphi} = \OOd(L^s),
 \end{equ}
\item if in addition $f$ contains {\em finitely} many Fourier modes, then
\begin{equs}
	 Qf = \OOd(L^{s-1})~.
\end{equs}
\end{enumerate}
\end{lemma}
\begin{proof}
	(a) and (b) are obvious. (c), (d) and (e) follow from Cauchy's theorem (see \lref{lem:poissonf1f2} and \cref{cor:poissonfg}).
(f) follows from iterating (e) and is proved in  \lref{lem:adl}. (g) is slightly more involved, as when $\ell_1 = \infty$ one must make sure that the domain does not shrink too much (see \pref{prop:67}). A proof of (h) is provided in \lref{lem:g10}. Finally,
(i) is proved in \pref{prop:chibound}: the idea that in the definition of $Q$, each term is divided by $\bmm \cdot \bI$, which scales like $L$ in an appropriate domain by \lref{lem:lhalf}.
\end{proof}

Each operation in \lref{lem:propsOO} imposes further restrictions on the parameters $r, \sigma, L_0$ in \eref{e:oo2} (see Appendix~\ref{sec:thebounds} for the details). This will, however, not be a problem since only a {\em finite number} of such operations will be needed.

As $\phi_i$ is not strictly speaking an admissible function (since it is not $\bphi$-periodic on $\complex^{2n})$, we need a supplementary lemma, which basically says that $\phi_i$ behaves like an $\OOd(L^0)$.
\begin{lemma}\label{lem:nonanalyticphii}
	Let $g = \OOd(L^{z})$. Then
\begin{equs}
\poiss{\phi_i}{g} = \OOd(L^{z-1})~,
\end{equs}
and
for all $1\leq \ell_0 \leq \ell_1 \leq \infty$,
\begin{equs}
	\sum_{\ell=\ell_0}^{\ell_1} \frac{1}{\ell!} \ad_{g}^\ell \phi_i = \OOd(L^{\ell_0(z-1)})~.
\end{equs}
\end{lemma}
\begin{proof}
The first statement follows from the fact that $\poiss{\phi_i}{g} = \partial_{I_i} g$ and \lref{lem:propsOO} (d). As a consequence, we have
\begin{equs}
	\sum_{\ell=\ell_0}^{\ell_1} \frac{1}{\ell!} \ad_{g}^\ell  \phi_i = \sum_{\ell=\ell_0-1}^{\ell_1-1} \frac 1 {\ell!}\ad_{g}^{\ell} \frac{\partial_{I_i}g}{\ell+1}~,
\end{equs}
which by \lref{lem:propsOO} (g) proves the second statement.
\end{proof}

\section{Inductive construction}\label{sec:general}

In this section, we describe the successive steps in which several
canonical transformations eliminate the non-resonant terms.
Throughout, we fix a cutoff $N_*$.
We will choose $N_*$ big enough so that the error terms $R\x j$
below are negligible powers of $L$. (We will see that $N_*=k$ is good enough.)
In a first reading, the reader may ignore this cutoff and the
corresponding remainders $R\x j$
(formally taking $N_* = \infty$).

We do the first two iterations explicitly in order to gain some intuition, and
then describe the general step.

\subsection{The first canonical transformation}\Label{sec:step2}

We start with the Hamiltonian $H\x0=h\x0+f\x0$ of the form of \eref{e:nn},
and decompose $f\x0$ according to \eref{e:RNR} and \eref{e:RNR2}.
This leads to
\begin{equs}
H\x0 = h\x0 + f\x{0,\R} + f\x{0,\NR} ~,
\end{equs}
where the explicit expression of $f\x{0,\R}$ and $f\x{0,\NR}$ was given in \eref{eq:decompF0}.
Obviously,  we have $f\x 0 = \OOd(L^0)$.
Following \sref{subs:eliminateNR} (with $f = f\x 0$ and $g = 0$), we let
\begin{equ}
  \chi\x0=Qf\x0~.
\end{equ}
We have $\poiss{h\x0}{\chi\x 0}=-f\x{0,\NR}$, and by \lref{lem:propsOO} (h),
\begin{equs}[eq:chi0m1]
\chi \x 0 = \OOd(L^{-1})~.
\end{equs}

By \eref{eq:iterationstep}, we thus find that the non-resonant term $f\x{0,\NR}$ is removed, and that
\begin{equs}
  H\x1 & \equiv e^{\chi\x0 }H\x0  = \,h\x0+f\x{0,\R} +  \sum_{\ell=1}^\infty \frac 1 {\ell!} \ad_{\chi^{(0)}}^\ell\left( f\x0 - \frac {f\NNR}{\ell+1}\right)  \\
& =h\x0+f\x{0,\R}+f\x1 + \RR\x1~,
\end{equs}
where
\begin{equs}[e:fr1]
	f\x1 &= \sum_{\ell=1}^{N_*-1} \frac 1 {\ell!} \ad_{\chi^{(0)}}^{\ell} \left( f\x0 - \frac {f\NNR}{\ell+1}\right) = \OOd(L^{-2})~,\\
\RR\x1 & = \sum_{\ell=N_*}^{\infty } \frac 1 {\ell!} \ad_{\chi^{(0)}}^{\ell} \left( f\x0 - \frac {f\NNR}{\ell+1}\right) = \OOd(L^{-2N_*})~.
\end{equs}
The estimates $\OOd(L^{-2})$ and $\OOd(L^{-2N_*})$ here come from \lref{lem:propsOO} (g), \eref{eq:chi0m1}, and the fact that $f\x0 = \OOd(L^0)$ (note that by \lref{lem:propsOO} (h), also   $f\x{0, \NR} = \OOd(L^0)$).

We now pause to make a series of observations and comments.

First, we put only finitely many
terms in $f\x1$, and the rest of the infinite series goes into $R\x1$.
The reason for this is that at the next step, we want to remove the non-resonant part of
$f\x 1$ by making another canonical transformation given by $\chi \x 1 = Qf\x1$, and for
this we cannot have infinitely many terms in $f\x1$. As mentioned above, $N_*$ will be chosen so that the remainder $R\x 1$ is small enough.

Secondly, and this is the essential feature, the coupling between the site $k$ and its neighbors, which was in $f\x {0,\NR}$, has been removed. Actually, the only interactions now involving the site $k$ appear in $f\x 1$ and $\RR \x 1$, so one can say that the site $k$ has been {\em decoupled} up to order $L^{-2}$ from the rest of the chain.

Next, we make
\begin{observation}\Label{lem:support}
The function $f\x1$ depends only on the variables $(\phi_i, I_i)$ for
$i\in \{k-2,\dots, k+2\}$. (We should actually write $\{\max(1,k-2),\dots, \min(k+2, n)\}$. To avoid burdening the notation, we will often omit to mention that the dependence is obviously restricted to the sites $\{1, \dots, n\}$.)
\end{observation}
This is seen as follows. By the structure of $f\x0$ as a sum of two-body interactions,
and by the definition of $\chi\x0$, we have that $\chi \x 0$ depends only on the sites
$k-1, k, k+1$. It is then easy to realize, using the structure of $f\x0$ again, that
\begin{equs}[eq:poisschi0f0]
\poiss{f\x0 - \frac {f\NNR}{\ell+1}}{\chi \x0}
\end{equs}
depends on the sites $k-2,\dots, k+2$. In fact, the dependence is extended
to the sites $k-2$ and $k+2$
due to the terms $U_{k-2}$ and $U_{k+1}$ in $f\x 0$:
the Poisson bracket $\poiss{U_{k-2}(\phi_{k-1}-\phi_{k-2})}{\chi \x 0}$
 couples the sites $k$ and $k-2$, and
the Poisson bracket $\poiss{U_{k+1}(\phi_{k+2}-\phi_{k+1})}{\chi \x 0}$
 couples the sites $k$ and $k+2$.
Moreover, starting with \eref{eq:poisschi0f0} and
taking more Poisson brackets with $\chi \x0$ does not extend the dependence
to more sites, which gives the observation above.

To summarize, at the end of this first canonical transformation,
we have replaced the non-resonant interactions
$f\x{0,\NR}$ with some smaller term $f\x 1 = \OOd(L^{-2})$ and a very
small remainder $R\x 1 = \OOd(L^{-2N_*})$. The site $k$ is then decoupled
up to order $L^{-2}$, but the potential $f\x 1$ makes it interact with
its next-to-nearest neighbors (whereas $H\x 0$ only featured nearest-neighbors
interactions).

\subsection{The second canonical transformation}

We first decompose
\begin{equs}
 H\x1 & = h\x0+f\x{0,\R}+f\x1 + \RR\x1 \\
&= h\x0+f\x{0,\R}+f\x{1,\NR}+f\x{1,\R} + \RR\x1~.
\end{equs}
We now want to remove $f\x{1,\NR}$. Following again the method of \sref{subs:eliminateNR}, this time with $f = f\x 1$ and $g = R\x 1 + f\x{0,\R}$, we define
\begin{equ}\Label{e:chi1}
  \chi\x1=  Q f\x1 ~,
\end{equ}
so that $\poiss{h\x0}{\chi\x1}=-f\ONR $ and $\chi \x1  = \OOd(L^{-3})$.
It is easy to see that \eref{eq:iterationstep} can be rearranged as
\begin{equs}
  H\x 2  \equiv e^{\chi\x {1} }H\x{1}&  =  h\x0 +f\x{0,\R} +f\x{1,\R} +f\x2 + R\x2 ~,
\end{equs}
with
\begin{equa}[e:fr2ndround]
	f\x 2 &= \sum_{\ell=1}^{N_*-1} \frac 1 {\ell!} \ad_{\chi\x{1}}^\ell\left(  f\x{0,\R} +f\x{1} - \frac {f\x{1, \NR}}{\ell+1} \right) = \OOd(L^{-4})~,\\
\RR\x{2} & =  \sum_{\ell=0}^\infty \frac 1 {\ell!} \ad_{\chi\x{1}}^\ell \RR\x{1} + \sum_{\ell=N_*}^\infty \frac 1 {\ell!} \ad_{\chi\x{1}}^\ell\left( f\x{0,\R} +f\x{1} - \frac {f\x{1, \NR}}{\ell+1}\right)= \OOd(L^{-2N_*})~,
\end{equa}
where we have used \lref{lem:propsOO} (g).

The potentials $f\x {0,\R}$ and $f\x{1,\R}$ do not involve $\phi_k$ by definition. Thus, since
only $f\x 2$ and $R\x 2$ involve $\phi_k$, we observe that the rotator $k$ is
now decoupled up to order $L^{-4}$.

Since $f\x 1$, and hence also $\chi \x 1$, involve only the sites $k-2,\dots, k+2$, we observe that
\begin{equs}[eq:poisschi1f1]
\poiss{ f\x{0,\R} +f\x{1} - \frac {f\x{1, \NR}}{\ell+1}}{\chi \x1}
\end{equs}
depends only on the variables of the sites $k-3,\dots, k+3$.
In fact, the support is only extended due to $f\x {0,\R}$:
the Poisson brackets $\poiss{U_{k-3}(\phi_{k-2} - \phi_{k-3})}{\chi \x 1}$
and  $\poiss{U_{k+2}(\phi_{k+3} - \phi_{k+2})}{\chi \x 1}$ couple
the site $k$ to the sites $k+3$ and $k-3$.
Taking further Poisson brackets of \eref{eq:poisschi1f1}
with $\chi \x 1$ does not extend the range of the interactions,
so we obtain

\begin{observation}\Label{lem:support}
The function $f\x 2$ involves only the sites $k-3,\dots, k+3$.
\end{observation}

\subsection{Canonical transformations for $j>1$}

We are now ready for the inductive step.
We start with
\begin{equs}
  H\x {j-1}  = h\x0+ \sum_{m=0}^{j-2}f\x{m,\R}+ f\x {j-1} + \RR\x{j-1}~,
\end{equs}
where
\begin{equa}[e:frordrejm1]
	f\x {j-1} &= \OOd(L^{-2(j-1)})~,\\
\RR\x{j-1} & = \OOd(L^{-2N_*})~,
\end{equa}
and where $f\x{j-1}$ depends only on the sites $k-j, \dots, k+j$.

 We decompose again
$$f\x{j-1} = f\x{j-1, \NR} + f\x{j-1, \R} 	~,
$$
and in order to remove $f\x{j-1, \NR}$, we follow \sref{subs:eliminateNR} with
\begin{equs}
f = f\x {j-1}~, \qquad 	g = \sum_{m=0}^{j-2}f\x{m,\R} + \RR\x{j-1}~.
\end{equs}
We thus define
\begin{equs}\label{e:chij}
 \chi\x{j-1}& =Q f\x {j-1} =  \OOd(L^{1-2j})~,
\end{equs}
so that $\{h\x0, \chi\x{j-1}\}  = - f\x{j-1, \NR}$.
This yields, by rearranging the terms of \eref{eq:iterationstep},
\begin{equs}[eq:defHxj]
  H\x j  \equiv e^{\chi\x {j-1} }H\x{j-1} =h\x0+ \sum_{m=0}^{j-1}f\x{m,\R} + f\x j + \RR\x{j}~,
\end{equs}
with
\begin{equa}[e:fr]
	f\x j &= \sum_{\ell=1}^{N_*-1} \frac 1 {\ell!} \ad_{\chi\x{j-1}}^\ell\left(  \sum_{m=0}^{j-2}f\x{m,\R} +f\x{j-1} - \frac {f\x{j-1, \NR}}{\ell+1} \right)~,\\
\RR\x{j} & =  \sum_{\ell=0}^\infty \frac 1 {\ell!} \ad_{\chi\x{j-1}}^\ell \RR\x{j-1} + \sum_{\ell=N_*}^\infty \frac 1 {\ell!} \ad_{\chi\x{j-1}}^\ell\left(  \sum_{m=0}^{j-2}f\x{m,\R} +f\x{j-1} - \frac {f\x{j-1, \NR}}{\ell+1}\right)~.
\end{equa}

At this point, we have an inductive definition of $f\x{j}$ and
$\RR\x{j}$ (one can set $R\x 0 = 0$).
By \lref{lem:propsOO} (g) and \eref{e:frordrejm1}, we find
\begin{equa}[e:frordre]
	f\x j &= \OOd(L^{-2j})~,\\
\RR\x{j} & = \OOd(L^{-2N_*})~.
\end{equa}

Moreover, since $\chi \x {j-1}$ only involves the variables $k-j, \dots, k+j$, we find
that
\begin{equs}[eq:chijm2chichm1]
\poiss{  \sum_{m=0}^{j-2}f\x{m,\R} +f\x{j-1} - \frac {f\x{j-1, \NR}}{\ell+1} }{\chi\x{j-1}}
\end{equs}
 depends only on the sites $k-j-1, \dots, k+j+1$. Indeed, the first argument of the Poisson brackets depends only on $k-j, \dots, k+j$, except for the term $f\x {0,\R}$, which
once again extends the range of the interactions by one due to $U_{k-j-1}(\phi_{k-j}-\phi_{k-j-1})$ and $U_{k+j}(\phi_{k+j+1}-\phi_{k+j})$. Taking further Poisson brackets of \eref{eq:chijm2chichm1} with $\chi \x {j-1}$ does not extend the dependence to new variables, so that we finally have
\begin{observation}
For $j\geq 1$, the function $f\x j$ depends only on the sites $\max(1,k-j-1), \dots, \min(k+j+1, n)$.
\end{observation}
(Note that we have to exclude the case $j=0$ from this statement, since the function $f\x 0$
involves all of the angles.)

The observation above as well as the orders in \eref{e:fr} are illustrated in
Table~\ref{table:illustrfchi} in the case $n=9$ and $k=6$.
\begin{table}[htb]
\centering
\begin{tabular}{l|lllllllll|l}
             & 1           & 2           & 3           & 4           & 5           & {\bf 6}         & 7           & 8           & 9           &                \\ \hline
$f\x 0$      & \textbullet & \textbullet & \textbullet & \textbullet & \textbullet & \textbullet & \textbullet & \textbullet & \textbullet & $\OOd(L^0)$    \\
$\chi \x 0$  &             &             &             &             & \textbullet & \textbullet & \textbullet &             &             & $\OOd(L^{-1})$ \\
$f\x 1$      &             &             &             & \textbullet & \textbullet & \textbullet & \textbullet & \textbullet &             & $\OOd(L^{-2})$ \\
$\chi \x 1$ &             &             &             & \textbullet & \textbullet & \textbullet & \textbullet & \textbullet &             & $\OOd(L^{-3})$ \\
$f\x 2$      &             &             & \textbullet & \textbullet & \textbullet & \textbullet & \textbullet & \textbullet & \textbullet & $\OOd(L^{-4})$ \\
$\chi \x 2$  &             &             & \textbullet & \textbullet & \textbullet & \textbullet & \textbullet & \textbullet & \textbullet & $\OOd(L^{-5})$ \\
$f\x 3$      &             & \textbullet & \textbullet & \textbullet & \textbullet & \textbullet & \textbullet & \textbullet & \textbullet & $\OOd(L^{-6})$ \\
$\chi \x 3$  &             & \textbullet & \textbullet & \textbullet & \textbullet & \textbullet & \textbullet & \textbullet & \textbullet & $\OOd(L^{-7})$ \\
$f\x 4$      & \textbullet & \textbullet & \textbullet & \textbullet & \textbullet & \textbullet & \textbullet & \textbullet & \textbullet & $\OOd(L^{-8})$ \\
$\chi \x 4$  & \textbullet & \textbullet & \textbullet & \textbullet & \textbullet & \textbullet & \textbullet & \textbullet & \textbullet & $\OOd(L^{-9})$
\end{tabular}
\caption{Illustration of the iterative construction in the case $n=9$ and $k=6$. The bullets indicate
the dependence on the sites.}\label{table:illustrfchi}
\end{table}

\section{Dynamics after $k-1$ iterations}\label{sec:k1}

We stop the process above after $k-1$ iterations. Defining
\begin{equs}[eq:HtildeORinit]
\tilde H &\equiv H\x{k-1}  =h\x0+ \sum_{m=0}^{k-2}f\x{m,\R} + f\x {k-1} + \RR\x{k-1}
\end{equs}
and letting
\begin{equ}
	N_* = k~,
\end{equ}
we find the estimates
\begin{equs}
	\sum_{m=0}^{k-2}f\x{m,\R} & = \OOd(L^{0})~,\\
	f\x {k-1}  & = \OOd(L^{2-2k})~,\\
	\RR\x{k-1}  & = \OOd(L^{-2k})~.
\end{equs}

\begin{remark}\label{rem:leftrightcoupled} In $\tilde H$, the site $k$ is decoupled from its neighbors up to order $L^{2-2k}$. However, if $k<n$, this does not imply that the subsystems $\{(\tilde I_i, \tilde \phi_i): i<k\}$ and $\{(\tilde I_i, \tilde \phi_i): i > k\}$
are decoupled from each other up to order $L^{2-2k}$. Indeed, $f\x {1,\R}$ contains terms of order $L^{-2}$ (\ie, of rather ``low'' order) coupling the sites $k-1$ and $k+1$.
\end{remark}

Tracing back the different steps, we have
\begin{equs}
\tilde H & = e^{\chi\x {k-2}}e^{\chi\x {k-3}}\cdots e^{\chi\x 0} H \\
&= H \circ \Psi~,
\end{equs}
where $\Psi$ is the composition of the time-one flow of the Hamiltonian fields
of $\chi \x0$, $\chi \x 1, \dots, \chi \x{k-2}$.
The coordinates $x = (\bI, \bphi)$ are then transformed into new coordinates
$\tilde x = (\tilde \bI, \tilde \bphi) = \Psi^{-1}(x)$,
given by
\begin{equs}[eq:xtildefctx]
\tilde x_a = \Psi_a^{-1}(x) \equiv 	e^{-\chi \x 0}e^{-\chi \x 1}\cdots e^{-\chi\x {k-2}}x_a~,
\end{equs}
where the argument of the $\chi \x i$ is $x$. Here and in the sequel, $x_a$
denotes any component of $(\bI, \bphi)$.
Conversely,
\begin{equs}
x_a = 	 \Psi_a(\tilde x) = e^{\chi\x {k-2}}e^{\chi\x {k-3}}\cdots e^{\chi\x 0} \tilde x_a~,
\end{equs}
where the argument of the $\chi \x i$ is now $\tilde x$, and the Poisson brackets are
taken with respect to the $\tilde x$, that is according to the rules
$\poiss{\tilde \phi_i}{ \tilde \phi_j} = \poiss{\tilde I_i}{ \tilde I_j} = 0$ and
$\poiss{\tilde \phi_i}{ \tilde I_j} = \delta_{i,j}$.

Obviously, we have
\begin{equs}[eq:dynamicsTilde]
\frac {\d}{\d t}\tilde x_a = \poiss{\tilde x_a}{\tilde H} - \gamma I_1 \partial_{I_1} \tilde x_a~,
\end{equs}
where again the Poisson brackets with respect to $\tilde x$ are used, and the argument of $\tilde H$ is $\tilde x$.

We now state and prove a series of technical results.

First, using recursively \lref{lem:propsOO} (g), the fact that all the $\chi \x j$ are at most $\OOd(L^{-1})$, and the fact that $I_i = \OOd(L^1)$, we have
\begin{equa}[eq:Tildeiimin]
\tilde I_i&= e^{-\chi \x 0}\cdots e^{-\chi\x {k-2}} I_i\\
& = I_i +  \OOd (L^{-1})~.
\end{equa}
Similarly, using \lref{lem:nonanalyticphii} we find that
\begin{equa}[eq:Tildephiimin]
\tilde \phi_i&= e^{-\chi \x 0}\cdots e^{-\chi\x {k-2}} \phi_i\\
& = \phi_i +  \OOd (L^{-2})~.
\end{equa}

\begin{remark}\label{rem:optimalpowers}
	The powers of $L$ in \eref{eq:Tildeiimin} and \eref{eq:Tildephiimin} are not optimal. Recalling that $\chi \x j = Qf \x j$ depends only on the sites $k-j-1, \dots, k+j+1$,
and that $\chi \x j = \OOd(L^{-2j-1})$, the interested reader can check that
for $i \neq k$ we actually have $\tilde I_i - I_i = \OOd(L^{1-2|k-i|})$ and $\tilde \phi_i - \phi_i = \OOd(L^{-2|k-i|})$ (see  Table~\ref{tbl:tableorders}). We shall not need the full extent of this result.
\begin{table}
\begin{center}
	\begin{tabular}{ccc}
\toprule
 & \multicolumn{2}{c}{Order} \\
\cmidrule(r){2-3}

$i$ & ~~~$\tilde I_i  - I_i$~~~ & ~~~$\tilde \phi_i- \phi_i$ ~~~ \\ \hline
1   & $L^{3-2k}$               &    $L^{2-2k}$            \\
2   & $L^{5-2k}$          &       $L^{4-2k}$                  \\
  $\dots$  &    $\dots$                 &     $\dots$         \\
$k-2$    &     $L^{-3}$                 &    $L^{-4}$                       \\
$k-1$    &     $L^{-1}$                 &    $L^{-2}$                       \\
$k$    &     $L^{-1}$                 &    $L^{-2}$                       \\
$k+1$    &     $L^{-1}$                 &    $L^{-2}$                       \\
$k+2$    &     $L^{-3}$                 &    $L^{-4}$                       \\
  $\dots$  &    $\dots$                 &     $\dots$         \\
\bottomrule
\end{tabular}
\end{center}
\caption{An overview of the relevant powers of $L$ in \rref{rem:optimalpowers}.}\label{tbl:tableorders}
\end{table}
\end{remark}

The following lemma and corollary guarantee that the domains from \dref{def:drs} and the  $\OOd$-notation
are sufficiently stable under our change of coordinates.

\begin{lemma} \label{lem:psimoinsid}
For all $r,\sigma, \rho > 0$ small enough, and all $r'<r, \sigma'<\sigma, \rho'<\rho$, the following holds:
for all $L$ large enough,
\begin{equs}[eq:inclusions1]
\Psi^\sharp(\DD_{L,r',\sigma'}) &\subset \DD_{L,r,\sigma}~,
\end{equs}
\begin{equs}[eq:inclusions2]
\Psi^\sharp(\BB_{L,\rho'}) &\subset \BB_{L,\rho}~,
\end{equs}
where $\Psi^\sharp$ stands for either $\Psi^{-1}$ or  $\Psi$.
\end{lemma}

\begin{proof}
By \eref{eq:Tildeiimin} and \eref{eq:Tildephiimin}, we can find some $r_0, \sigma_0, L_0, C > 0$ such that for all $L\geq L_0$ and all $x\in \DD_{L, r_0, \sigma_0}$, we have
\begin{equs}
\| \Psi^{-1}(x) - x \| \leq \frac C L~.
\end{equs}
Thus, for any $0 < \varepsilon < \varepsilon '  < \varepsilon '' $, we have some $L_0' \geq L_0$ such that for all $L\geq L_0'$ and all $x\in \DD_{L, r_0, \sigma_0}$, the following inclusions of complex balls hold:
\begin{equs}[eq:cplxballs]
	B(x, \varepsilon) \subset \Psi^{-1}(B(x, \varepsilon')) \subset B(x, \varepsilon'')~.
\end{equs}

The second inclusion immediately proves that $\Psi^{-1}(\DD_{L,r',\sigma'}) \subset \DD_{L,r,\sigma}$ for large enough $L$, provided that $r< r_0$ and $\sigma < \sigma_0$. Under the same conditions, the first inclusion in \eref{eq:cplxballs}, which also reads $\Psi(B(x, \varepsilon) )\subset B(x, \varepsilon')$, proves that $\Psi(\DD_{L,r',\sigma'}) \subset \DD_{L,r,\sigma}$.

Next, \eref{eq:inclusions2} follows in a similar manner, using \rref{rem:inclusionBBDD} and the fact that $\Psi$ maps real points to real points, so that a real equivalent of \eref{eq:cplxballs} holds.
\end{proof}

The lemma above and the definition of $\OOd(L^{s})$ immediately imply
\begin{corollary} \Label{cor:psimoinsidcor} If $f=\OOd(L^s)$, then
\begin{equs}[eq:fcomppsi]
f \circ \Psi^\sharp = \OOd(L^s)~,
\end{equs}
where $\Psi^\sharp$ stands for either $\Psi^{-1}$ or  $\Psi$.

\end{corollary}

Thanks to this corollary, we will not need to specify
whether the $\OOd$ are to be expressed in terms of $x$ or $\tilde x$.
Note that the corollary can also be viewed as a consequence of
\lref{lem:propsOO} (g) and the fact that $f \circ \Psi = e^{\chi\x {k-2}}e^{\chi\x {k-3}}\cdots e^{\chi\x 0}f $ 
and $f \circ \Psi^{-1} =  	e^{-\chi \x 0}e^{-\chi \x 1}\cdots e^{-\chi\x {k-2}} f $.

A series of definitions and technical results is necessary in order to understand the dynamics of $\tilde x$.

\begin{definition}\label{def:OOdRNRfin}
Let $f = \OOd(L^{s})$. We say that $f = \OOd_{\R}(L^{s})$ if the Fourier series \eref{eq:fgeneric} contains only resonant terms (\ie, $f = f^\R$), and similarly we say that $f=\OOd_{\NR}(L^{s})$ if
the Fourier series contains only non-resonant terms (\ie, $f = f^\NR$). Moreover, we denote by $\OOd^\fin$, $\OOd^\fin_{\NR}$ and $\OOd^\fin_{\R}$ the remainders whose Fourier series \eref{eq:fgeneric} contains only {\em finitely} many Fourier modes.
\end{definition}

Since $f\x{m,\R} = \OOd_{\R}^\fin(L^{-2m})$, observe that \eref{eq:HtildeORinit} reads
\begin{equs}[eq:HtildeOR]
  \tilde H(\tilde x)  &= \sum_{i=1}^n \frac{\tilde I_i^2}{2} + f\x{0, \R}(\tilde \bphi)+ \OOd_{\R}^\fin(L^{-2})	+ \OOd(L^{2-2k})~.
\end{equs}
We observe also that by \eref{e:chij} and the definition of $\chi \x j$,
\begin{equs}\Label{e:chijNR}
 \chi\x{j} = \OOd_\NR^\fin(L^{-2j-1})~.
\end{equs}

\begin{lemma}\label{lem:lemp1}We have
\begin{equs}
I_1 = \tilde I_1 + P_1 (\tilde x)~,\qquad \text{with }\quad
P_1 =  \OOd(L^{3-2k}) ~.
\end{equs}
More precisely, we have the decomposition
\begin{equs}[eq:decompositionP1]
P_1 =  -\partial_{ \phi_1}\chi\x{k-2} + \OOd(L^{5-4k})~.
\end{equs}
\end{lemma}
\begin{proof}
Since the only $\chi \x j$ that involves the site $1$ is $\chi\x{k-2}$, we find
\begin{equa}[eq:decompI1p1]
   I_1&= e^{\chi\x{k-2}} \cdots e^{\chi\x{0}}\tilde  I_1= e^{\chi\x{k-2}} \tilde I_1~,
\end{equa}
where it is understood that the argument of $\chi\x{j}$ is $\tilde x$. Since $\chi \x{k-2} = \OOd(L^{3-2k})$, both statements immediately follow.
\end{proof}

We will  need bounds on the time evolution of $\tilde x$.  Because of the changes of variables, the dissipation no longer acts only on the first rotator. To compute the effect of dissipation on each $\tilde x_a$, we need to understand how they depend on the original $I_1$, onto which the dissipation acts.

\begin{lemma}\label{lem:deriveewrtI1}
For $i=1, 2, \dots, n$, we have
\begin{equs}[eq:derphiitildei1]
	\frac{\partial  \tilde I_i}{\partial I_1} &= \delta_{1,i} +  \OOd(L^{2-2k})~,\\
	\frac{\partial  \tilde \phi_i}{\partial I_1} &= \OOd(L^{1-2k})~.
\end{equs}
\end{lemma}
\begin{proof}
Since only $\chi\x{k-2}$ involves the site 1,
 and $\chi\x{k-2} = \OOd(L^{3-2k})$, we observe that \eref{eq:xtildefctx} leads to
\begin{equs}
	\frac{\partial  \tilde I_i}{\partial I_1} &=  e^{-\chi\x{0}}\cdots e^{-\chi\x{k-3}} \frac{\partial }{\partial I_1}\left(e^{-\chi\x{k-2}} I_i\right)\\
& =   e^{-\chi\x{0}}\cdots e^{-\chi\x{k-3}} \frac{\partial }{\partial I_1} [I_i + \OOd(L^{3-2k})]\\
& =  e^{-\chi\x{0}}\cdots e^{-\chi\x{k-3}} [\delta_{1,i} + \OOd(L^{2-2k})] = \delta_{1,i} + \OOd(L^{2-2k})~.
\end{equs}
(Recall that the Poisson bracket of  $\delta_{1,j}$ and any function is zero.)
For $\tilde \phi_i$ we find similarly, using \lref{lem:nonanalyticphii}, that
\begin{equs}
	\frac{\partial  \tilde \phi_i}{\partial I_1} &=  e^{-\chi\x{0}}\cdots e^{-\chi\x{k-3}} \frac{\partial }{\partial I_1}\left(e^{-\chi\x{k-2}} \phi_i\right)\\
& =   e^{-\chi\x{0}}\cdots e^{-\chi\x{k-3}} \frac{\partial }{\partial I_1} [\phi_i + \OOd(L^{2-2k})]\\
& =  e^{-\chi\x{0}}\cdots e^{-\chi\x{k-3}} [ 0 +\OOd(L^{1-2k})] = \OOd(L^{1-2k})~,
\end{equs}
which completes the proof.
\end{proof}

\begin{lemma}\label{lem:eqItilde}
We have:
\begin{equs}[eq:dynamicstildeItildephi]
	\dt  \tilde I_{i}& = -\frac{\partial f\x{0,\R}}{\partial \phi_i}(\tilde \bphi)+ (1-\delta_{i,k})\OOd_{\R}^\fin(L^{-2}) + \OOd(L^{3-2k})
 -\gamma \delta_{1,i} \tilde I_1~,\\
	\dt  \tilde \phi_{i} &= \tilde I_i + \OOd_{\R}^\fin(L^{-3}) + \OOd(L^{2-2k})~.
\end{equs}
\end{lemma}
\begin{proof}
We have
\begin{equs}
\frac{\d}{\d t} \tilde  I_i &= \{\tilde I_i, \tilde H\}	 - \gamma I_1 \frac{\partial }{\partial I_1} \tilde I_i~.
\end{equs}
First we note that by \eref{eq:HtildeOR},
\begin{equs}
\{\tilde I_i, \tilde H\}	 &=  -\frac{\partial f\x{0,\R}}{\partial \phi_i}(\tilde \bphi)+ (1-\delta_{i,k})\OOd_{\R}^\fin(L^{-2}) + \OOd(L^{2-2k})~,
\end{equs}
where the $(1-\delta_{i,k})$ comes from the fact that $\partial_{\phi_k} \OOd_\R^\fin(L^{-2}) = 0$, since the remainder is resonant and hence independent of $\phi_k$.
Moreover, we get  by \eref{eq:derphiitildei1} and \lref{lem:lemp1} that
\begin{equs}
 I_1 \frac{\partial }{\partial I_1} \tilde I_i &= 	(\tilde I_1 + \OOd(L^{3-2k}))(\delta_{1,i} +  \OOd(L^{2-2k})) = \tilde I_1 \delta_{1,i} + \OOd(L^{3-2k})~.
\end{equs}
This completes the proof of the first line of \eref{eq:dynamicstildeItildephi}.

In the same way,
\begin{equs}
\frac{\d}{\d t} \tilde  \phi_i &= \{\tilde \phi _i, \tilde H\}	 - \gamma I_1 \frac{\partial }{\partial I_1} \tilde \phi_i~.
\end{equs}
Now, using again \eref{eq:HtildeOR} leads to
\begin{equs}
\{\tilde \phi_i, \tilde H\}	 &= \tilde I_i  + \OOd_\R^\fin(L^{-3}) + \OOd(L^{1-2k}) ~.
\end{equs}
Since by \lref{lem:lemp1} and \eref{eq:derphiitildei1},
\begin{equs}
 I_1 \frac{\partial }{\partial I_1} \tilde \phi_i &= 	(\tilde I_1 + \OOd(L^{3-2k})) \OOd(L^{1-2k}) = \OOd(L^{2-2k})~,
\end{equs}
the proof of the second line of \eref{eq:dynamicstildeItildephi} is complete.
\end{proof}

Note that since $f\x{0,\R}$ does not involve $\phi_k$ by construction, \eref{eq:dynamicstildeItildephi} implies that
\begin{equ}[eq:specialcasedtIk]
		\dt  \tilde I_{k} = \OOd(L^{3-2k}) ~.
\end{equ}

\subsection{Long-time stability}

In the sequel, $\alpha > 0$ and $T=\alpha L^{2k-3}$ are as in \tref{t:main}. We prove a statement
slightly stronger than part (i) of \tref{t:main}, since it applies to both $x(t)$ and $\tilde x(t)$ (such a generalization will be needed to prove  part (ii) of \tref{t:main}).

\begin{proposition}\label{p:wasGronwall}Let $\rho>0$. Then, there exist constants $L_0, \rho^*>0$ such that for all
$L\geq L_0$ and all initial conditions $x(0) \in \BB_{L,\rho}$ we have both
\begin{equs}
x(t) \in \BB_{L, \rho^*}, \quad \text{and} \quad \tilde x(t) \in \BB_{L, \rho^*}~, \quad 0\leq t \leq T~.
\end{equs}

\end{proposition}
\begin{proof}
We recall that since we consider real initial conditions $x\in \BB_{L, \rho}$, then both $x(t)$ and $\tilde x(t)$ remain real for as long as they are defined. In particular, we can determine whether the orbit remains in $\BB_{L, \rho^*}$ by just tracking the behavior of the action variables.

We define
\begin{equs}[eq:Htildem]
\tilde H_* = \tilde H - \frac{\tilde I_k^2}2 = \sum_{i\neq k} \frac{\tilde I_i^2}{2} + \OOd(L^0)~.
\end{equs}

Since $\frac{\d}{\d t} \tilde H = -\gamma I_1^2$, we have by \eref{eq:dynamicsTilde}
\begin{equs}[eq:derhtildet]
	\frac{\d}{\d t} \tilde H_* &=  -\frac{\d}{\d t}\frac{\tilde I_k^2}2 - \gamma I_1^2 = -\tilde I_k \left(\frac{\partial \tilde H}{\partial { \phi_k}}(\tilde x) - \gamma I_1 	\frac{\partial  \tilde I_k}{\partial I_1}  \right) - \gamma I_1^2\\
& \leq -\tilde I_k  \frac{\partial \tilde H}{\partial { \phi_k}}(\tilde x) + \frac \gamma 4 	\left(\tilde I_k\frac{\partial  \tilde I_k}{\partial I_1}  \right)^2\\
&= \tilde I_k \OOd (L^{2-2k}) + \tilde I_k^2 \OOd (L^{4-4k})  = \OOd(L^{3-2k})~,
\end{equs}
where the second line follows from Young's inequality,
and the third line follows from \lref{lem:deriveewrtI1}. (Note that the $\OOd (L^{2-2k}) $ in the last line comes from the fact that in \eref{eq:HtildeOR}, the resonant terms do not depend on $\phi_k$ by definition.)

By \eref{eq:Htildem}, \eref{eq:derhtildet} and \eref{eq:specialcasedtIk}, we obtain that  there exist $r, \sigma, C>0$ such that
for all large enough $L$, we have on the set $\DD_{L, r, \sigma}$\,,
\begin{equs}[eq:estimatestst]
&0 \leq  \sum_{i\neq k} \frac{\tilde I_i^2}{2} \leq  \tilde H_* + C ~, \\
&	\frac{\d}{\d t} \tilde H_* \leq C L^{3-2k}~,\\
&	\left|\frac{\d}{\d t} \tilde I_k \right|  \leq C L^{3-2k}~.
\end{equs}
Now, provided $L$ is large enough, we have $\tilde x(0) \in \BB_{L, 2\rho}$ by \lref{lem:psimoinsid}.
Possibly increasing the value of $C$, we further require $L$ to be large enough so that $$\sup_{x\in \BB_{L, 2\rho}} \tilde H_*(x) \leq C~,$$
which is possible by \rref{rem:inclusionBBDD} and the definition of $\tilde H_*$.
Choose now
\begin{equ}
	\rho^* > 2\max( 2\rho + \alpha C, \sqrt{2(2+\alpha)C})~.
\end{equ}
By further restricting the allowed values of $L$, we have
by \rref{rem:inclusionBBDD}
that $\BB_{L, \rho^*} \subset \DD_{L, r, \sigma}$.
Consider now the first exit time
\begin{equ}
	t^* = \inf\{t \geq 0 ~:~ \tilde x(t) \notin \BB_{L, \rho^*} \}~,
\end{equ}
with the convention $\inf \{\emptyset\}=\infty$.
For all $t\leq t^*$ the estimates in \eref{eq:estimatestst} hold,
and thus  for all $t\leq \min(T, t^*)$ we have
\begin{equs}
	|\tilde I_k(t) - L| \leq |\tilde I_k(0)-L| + tC L^{3-2k} \leq  2\rho + \alpha C < \frac{\rho^*}2~,
\end{equs}
and
\begin{equs}
	\tilde H_*(\tilde x(t)) \leq \tilde H_*(\tilde x(0)) +  tC L^{3-2k} \leq (1+\alpha)C~,
\end{equs}
which further implies that for all $i\neq k$,
\begin{equs}
	|\tilde I_i(t)| \leq \sqrt{2(\tilde H_*(\tilde x(t))+C)} \leq \sqrt{2(2+\alpha)C} < \frac{\rho^*}2~.
\end{equs}

This implies that for all $t\leq \min(T, t^*)$ we have $\tilde x(t) \in \BB_{L, \rho^*/2}$.
By continuity of $\tilde x(t)$ with respect to $t$, this implies that $t^* > T$, so that actually
$\tilde x(t) \in \BB_{L, \rho^*/2}$ for all $t\leq T$. But then by \lref{lem:psimoinsid},
we conclude that for all large enough $L$ and all $t\leq T$, $x(t) \in \BB_{L, \rho^*}$.
This completes the proof.
\end{proof}

\section{Estimating the dissipation}\label{sec:estimatingdiss}

We now fix once and for all $\rho, \alpha > 0$ as in \tref{t:main}.
We fix also $\rho^*$
as in  \pref{p:wasGronwall}, so that for all large enough $L$, and all $x(0) \in \BB_{L, \rho}$,
\begin{equ}[eq:remaininBLrp]
	x(t) \in \BB_{L, \rho^*}, \quad \tilde 	x(t) \in \BB_{L, \rho^*}, \quad 0 \leq t \leq T \equiv \alpha L^{2k-3}~.
\end{equ}

\noindent {\bf Convention.} In the sequel, and without further mention,
we consider only initial conditions $x(0) \in \BB_{L, \rho}$,
where $L$ is large enough so that \eref{eq:remaininBLrp} holds.

We will at several (but finitely many) occasions
increase the lower bound on the allowed values for $L$.
With this in mind,  we now introduce a weaker notion
of $\OOs$, which we will use alongside
the $\OOd$ introduced in \dref{def:OOd}  (recall also \rref{rem:distsets}).
\begin{definition}\Label{def:newOOs}
Given a function $f$ defined on $\BB_{L, \rho^*}$ for
all large enough $L$, we say that $f=\OOs(L^s)$ if there exists a constant
$c$ such that for all large enough $L$,
we have $\sup_{x\in \BB_{L, \rho^*}}|f(x)| \leq c L^s$. The notations $\R, \NR$ and $\fin$ bare the same meaning as in \dref{def:OOdRNRfin}.
\end{definition}

We shall sometimes also write $\OOs(T^r L^s)$, which means $\OOs(L^{s + r(2k-3)})$ since
$T = \alpha L^{2k-3}$. Moreover, with \eref{eq:remaininBLrp} in mind,
we extend this notation to functionals of the trajectory:
given $f = \OOs(L^s)$, we write for example $f(x(t)) = \OOs(L^s)$ when $t\leq T$, or
$\int_0^T f(x(t)) \d t = \OOs(T L^s)$.

By \rref{rem:inclusionBBDD}, any $f$ that is $\OOd(L^s)$ is also $\OOs(L^s)$.
The main difference between $\OOd$ and $\OOs$ is that for all $i\neq k$ we have
\begin{equs}
	I_i = \OOs(L^0)~,
\end{equs}
whereas we had $I_i = \OOd(L)$.

We want to estimate $H(x(T))-H(x(0))$ when $T =  \alpha L^{2k-3}$, with
$\alpha >0$.
Since the dissipation happens only in the variable $I_1$, we have
\begin{equ}\label{e:timeintegral}
  H(x (T))-H(x(0))= -\gamma \int_0^T I_1^2(t) \,\d t~.
\end{equ}
In the new variables and using \lref{lem:lemp1}, \eref{e:timeintegral} takes the form
\begin{equa}[e:dissipative]
  H(x(T))-H(x(0))&=-\gamma \int_0^T {{\tilde I_1^2 }}(t) \, \d t\\
&\quad -\gamma
  \int_0^T P_1^2 (\tilde x(t))\,\d t\\
&\quad -2\gamma \int_0^T \tilde  I_1 (t)\,   P_1(\tilde x(t))\,\d t~.
\end{equa}

We can say the following about the three contributions above.
\begin{itemize}
	\item The first integral has a negative sign, but our analysis gives little control over it. Our numerical experiments indicate that the integrand is only significant for a short, initial transient, and that its contribution to the dissipation in the quasi-stationary state is negligible.
\item The second integral also has a negative sign, and is of order $\OOs(TL^{6-4k}) = \OOs(L^{3-2k})$. This is where
the dissipation rate in \tref{t:main} (ii) comes from. In \sref{sec:dissipnotzero} we will show
that under Assumptions \ref{d:22} and \ref{d:23}, the integral $\int_0^T P_1^2(\tilde{x}(t))\, \d t$
is bounded {\em below} by a constant times $L^{3-2k}$, which is crucial.
\item The third integral has no sign, and naive dimensional analysis suggests
it is $\OOs(TL^{3-2k}) = \OOs(L^0)$.
However, due to its oscillatory nature, this integral
is in fact much smaller.
In \sref{sec:partii}, we use integration by parts (homogenization)
to show that it
can essentially be reduced to boundary terms of order $\OOs(L^{2-2k})$, plus
higher order corrections, so that it is negligible compared to the first two integrals in \eref{e:dissipative}.
\end{itemize}

\begin{remark}\label{rem:choiceT}The choice $T\sim L^{2k-3}$ is made so that the second integral in \eref{e:dissipative},
which scales like $TL^{6-4k}$, dominates the boundary terms of order  $L^{2-2k}$ coming from the third
integral.
\end{remark}

The next two subsections, which deal respectively with the second and third terms in \eref{e:dissipative},
will lead to the proof of \tref{t:main}.

\subsection{Dissipation is not zero}\label{sec:dissipnotzero}

The main result of this subsection is the following lower bound on the
absolute value of the second term in \eref{e:dissipative}:
\begin{proposition}\label{p:p1not0}
  There exists a constant $c>0$ such that for all $L$ large enough,
  and $T=\alpha L^{2k-3}$,
  \begin{equ}[eq:integralep1sq]
    \int_0^T P^2_1(\tilde x(t))\, \d t  \geq  \frac{\alpha c}{L^{2k-3}}~.
  \end{equ}
\end{proposition}

Before we start with the proof of \pref{p:p1not0}, we need
some auxiliary material.

Recalling the definition of $Q$ in \eref{eq:defQ}, we immediately have
\begin{equs}[eq:commutatdT]
\partial_{I_i} (Q f) & =-\partial_{I_i} \left(\i\sum_{\bmm\in\NN^{\NR}}
  \frac{f_\bmm(\bI)}{\IM}e^{\i \,\mphi}\right)	\\
& =\i\sum_{\bmm\in\NN^{\NR}}
  \left(-\frac{\partial_{I_i} f_\bmm(\bI)}{\IM} + \frac{\mu_i f_\bmm(\bI)}{(\IM)^2}\right)e^{\i \,\mphi}\\
& = (Q \partial_{I_i} - Q^2 \partial_{\phi_i}) f~.
\end{equs}
In the sequel, we often omit the argument $\phi_{j+1}-\phi_{j}$ of $U_j$ and its derivatives.

\begin{lemma}\label{eq:commfj}
We have
\begin{equs}[eq:derphifj]
\partial_{\phi_{k-j-1}} f\x j = U''_{k-j-1} Q^2 \partial_{\phi_{k-j}} f\x {j-1} + \OOd^\fin(L^{-2j-1})~.
\end{equs}
\end{lemma}
\begin{proof}
Recall that by \eref{e:fr},
\begin{equs}
f\x j &= \left\{ \sum_{s=0}^{j-2}f\x{s,\R} +f\x{j-1} - \frac {f\x{j-1, \NR}}{2} , \chi\x{j-1}\right\} + \OOd^\fin(L^{-2j-1}) ~,
\end{equs}
where the remainder contains finitely many Fourier modes, since $f\x j$ does.
The only term in the Poisson bracket involving the site $k-j-1$ is actually in $f\x {0,\R}$ here, since all other terms depend only on the variables $k-j, k-j+1, \dots$. Thus,
\begin{equs}
	\partial_{\phi_{k-j-1}} f\x j &= \left\{\partial_{\phi_{k-j-1}} f\x {0,\R} , \chi\x{j-1}\right\} + \OOd^\fin(L^{-2j-1})~.
\end{equs}
Observing that $\partial_{\phi_{k-j-1}} f\x {0,\R} =
U'_{k-j-2}(\phi_{k-j-1}-  \phi_{k-j-2}) -U'_{k-j-1}(\phi_{k-j} -\phi_{
  k-j-1})$ and that the first contribution has vanishing  Poisson bracket with $\chi \x{j-1}$, we obtain
\begin{equs}
	\partial_{\phi_{k-j-1}} f\x j &= \left\{ -U'_{k-j-1} , \chi\x{j-1}\right\} + \OOd^\fin(L^{-2j-1}) \\
& = - U''_{k-j-1}  \partial_{I_{k-j}} \chi\x{j-1}+ \OOd^\fin(L^{-2j-1})~.
\end{equs}
But now, $ \partial_{I_{k-j}} \chi\x{j-1} =  \partial_{I_{k-j}} Q f\x{j-1} = (Q \partial_{I_{k-j}} - Q^2 \partial_{\phi_{k-j}})f \x {j-1}$ by \eref{eq:commutatdT}. But by its definition, $f \x {j-1}$ depends on $\phi_{k-j}$ but not  on $I_{k-j}$ (because it is built from some Poisson brackets between $\chi\x{j-2}$ and some forces, none of which involve $I_{k-j}$). Thus, we indeed obtain \eref{eq:derphifj}.
\end{proof}

Now, using \eref{eq:decompositionP1}, the fact that $\partial_{\phi_1}$ commutes with $Q$,
and then \lref{eq:commfj} with $j=k-2$ leads to
\begin{equs}
	P_1 &= -\partial_{\phi_1} \chi\x{k-2} +\OOd(L^{5-4k})\\
&= -Q \partial_{\phi_1} f\x {k-2} + \OOd(L^{5-4k}) \\
& =  -Q \left( U''_{1} Q^2 \partial_{\phi_2} f\x{k-3}  + \OOd^\fin(L^{3-2k})\right) + \OOd(L^{5-4k})\\
& =  -Q  U''_{1} Q^2 \partial_{\phi_2} f\x{k-3}  + \OOd(L^{2-2k}) ~.
\end{equs}
Using \lref{eq:commfj} again repeatedly and finally that $\partial_{\phi_k} f\x 0 = U'_{k-1}$, we obtain
\begin{equs}[eq:P1pmUUU]
P_1 &= -Q U''_{1} Q^2 U''_{2} Q^2 \cdots U''_{k-2 }Q^2 U'_{k-1} + \OOd(L^{2-2k})~.
\end{equs}

We now show that $P_1$ is well approximated
by a simple explicit formula. To construct this formula,
consider first \eref{eq:P1pmUUU}
for the special choice $I_k=L$ and $I_i=0$ for all $i\neq k$. Then $Q$ has the effect of dividing
each term of the Fourier series
by $\i \mu_k L$, \ie, it integrates with
respect to $\phi_k$ and divides by $L$. Since $U_{k-1}$ is the only
factor that depends on $\phi_k$, the operator $Q$ just integrates
$U_{k-1}$ and divides by $L$. Thus, since the operator $Q$ appears $2k-3$ times,
we expect $P_1$ to be well approximated by $M_1(\bphi)/L^{2k-3}$ with
\begin{equs}[eq:defM1] M_1(\bphi)=  -{U''_{1}U''_{2}\cdots U''_{k-2} G }~,
\end{equs}
where $G$ is the unique solution on $\torus$ of ${\d^{2k-4}}G(\phi)/\d \phi^{2k-4} = U_{k-1}(\phi)$
satisfying
$\int_0^{2\pi} G(\phi)\, \d \phi = 0$. The next lemma shows that the error made by this approximation is only $\OOs(L^{2-2k})$.
\begin{lemma}\label{lem:realboundP1}
We have
\begin{equs}[eq:P1avecR]
P_1 = \frac{M_1(\bphi)}{L^{2k-3}} +\OOs(L^{2-2k})~.
\end{equs}
\end{lemma}
\begin{proof}
Consider the Fourier series of $U_i$, written
\begin{equ}
U_i(\phi_{i+1}-\phi_i) = \sum_{\bmm \in \NN_i} \hat U_i(\bmm )e^{\i \bmm \cdot  \bphi}~,
\end{equ}
where $\NN_i$ is a finite (by \aref{d:22}) set of vectors $\bmm \in \integer^n$ satisfying
$\mu_j = 0$ for $j\notin \{i, i+1\}$, and $\mu_{i+1} = -\mu_i$. Note that in particular
$
U_i'(\phi_{i+1}-\phi_i) =\i  \sum_{\bmm \in \NN_i} \mu_{i+1} \hat U_i(\bmm )e^{\i \bmm \cdot  \bphi}~,
$
and similarly for $U_i''$ with one more factor of $\i \mu_{i+1}$.
In such notation, observe that \eref{eq:P1pmUUU} reads
\begin{equs}[eq:summumkm1]
P_1 &= - \sum_{\bmm\x 1, \dots, \bmm\x {k-1}}  \frac{e^{\i  \bmm\xx {1}\cdot \bphi }}{\mu \x{k-1}_k}  \frac{(\mu \x 1_2)^2\hat U_1(\bmm \x 1)}{\bmm \xx {1} \cdot \bI} \prod_{i=2}^{k-1} \frac{(\mu \x i_{i+1})^2\hat U_i(\bmm \x i)}{(\bmm\xx {i} \cdot \bI)^2 } + \OOd(L^{2-2k})~,
\end{equs}
where we sum over all $\bmm \x i \in \NN_i$ for $i=1, \dots, k-1$, and where we write
\begin{equ}
\bmm \xx {j} = \sum_{i=j}^{k-1}\bmm \x {i}~.
\end{equ}

Now, observe that
\begin{equ}[eq:avanttaylor]
 \frac{1}{\bmm \xx {1} \cdot \bI} \prod_{i=2}^{k-1} \frac{1}{(\bmm\xx {i} \cdot \bI)^2 } =  \frac 1 {L^{2k-3}} \frac{1}{\mu \xx 1_k} \prod_{i=2}^{k-1} \frac{1}{(\mu\xx i_k)^2 } + \OOs(L^{2-2k})~,
\end{equ}
as Taylor's theorem and the definition of $\BB_{L, \rho^*}$ show. (Note that here we only have an $\OOs(L^{2-2k})$ and not an $\OOd(L^{2-2k})$.)

Moreover, observe that for all $i=1, \dots, k-1$, we have $\mu \xx i_k = \mu \x {k-1}_k$,
which is just a manifestation of the fact that $U_{k-1}$ is the only interaction potential involving $\phi_k$. By this observation,  \eref{eq:avanttaylor}, and the fact that the sum in \eref{eq:summumkm1} is finite, we obtain
\begin{equs}
P_1 &=   -\frac 1{L^{2k-3}}\sum_{\bmm\x 1, \dots, \bmm\x {k-1}} {e^{\i  \bmm\xx 1\cdot \bphi }} \frac{\hat U_{k-1}(\bmm \x{k-1})}{(\mu\x{k-1}_k)^{2k-4}} \prod_{i=1}^{k-2}(\mu \x i_{i+1})^2\hat U_i(\bmm \x i) + \OOs(L^{2-2k})~,
\end{equs}
from which it is easy to identify \eref{eq:P1avecR}.
\end{proof}

Looking at \eref{eq:dynamicstildeItildephi}, we expect that when $L$ is large, $\tilde x(t)$ will be well approximated by the solution $\overline x(t)$ of the following {\em decoupled ODE}:
\begin{equs}[eq:ODEdec1]
	\dt  \overline {I}_{i}& = -\frac{\partial f\x{0,\R}}{\partial \phi_i}(\overline {\bphi})
 -\gamma \delta_{1,i} \overline {I}_1~,\\
	\dt  \overline \phi_{i} &= \overline I_i~,
\end{equs}
with $\overline x(0) = \tilde x(0)$.
We call this system {\em decoupled} since the site $k$ does not interact at all with the other sites (recall that $f\x{0,\R}$ contains all the interaction potentials $U_i$ except those that involve the site $k$, see \eref{eq:decompF0}). In fact, the subsystems $(\overline\phi_{\ell}, \overline I_\ell, \dots, \overline\phi_{k-1}, \overline I_{k-1})$, $(\overline\phi_{k}, \overline I_{k})$, and $(\overline\phi_{k+1}, \overline I_{k+1}, \dots,\overline\phi_{n}, \overline I_{n})$ (if $k<n$) are isolated from each others, and in particular $\overline I_k$ remains constant.

The next proposition shows that along the trajectories of the decoupled system, the integral $\int_0^1 M_1^2(\overline{\bphi}(t))\,\d t$ is bounded below. This and the fact that $\tilde x(t)$ is close to $\overline x(t)$ will lead to the proof of \pref{p:p1not0}.

\begin{lemma}\label{lem:solutionanalytique}
The solution to \eref{eq:ODEdec1} is a real-analytic function of $t$ for all times.
\end{lemma}
\begin{proof}
It follows from the Cauchy-Kovalevskaya theorem that the solution is
locally real analytic in time.
Since the energy of the decoupled system is non-increasing, the solution cannot blow up, and
the Picard existence and uniqueness theorem guarantees that the solution is real-analytic for all times.
\end{proof}

\begin{proposition}\label{prop:M1t01}
Under the decoupled dynamics \eref{eq:ODEdec1}, there is a constant
 $\overline c_*>0$ such that if $L$ is large enough, then for any initial
condition
$\overline x(0)=(\overline\bI(0),\overline{\bphi} (0)) \in \BB_{L, \rho^*}$, we have
\begin{equs}[eq:intm2gqca]
\int_0^1 M_1^2(\overline{\bphi}(t))\,\d t  \geq  \overline{c}_*~.
\end{equs}
\end{proposition}
\begin{proof}
Let
\begin{equs}
K(t) = U_1''\bigl(\overline{\phi}_2(t)-\overline{\phi}_1(t)\bigr)\cdot
U_2''\bigl(\overline{\phi}_3(t)-\overline{\phi}_2(t)\bigr) \cdots U_{k-2}''\bigl(\overline{\phi}_{k-1}(t)-\overline{\phi}_{k-2}(t)\bigr)~.
\end{equs}

First, observe that there is a constant $C$, independent of $L$, such that for all $\overline  x(0) \in \overline{\BB}_{L, \rho^*}$ (the closure of $\BB_{L, \rho^*}$)
and all $t\in [0,1]$ we have $|K(t)| \leq C$ and $|\dot K(t)| \leq C$.

Next, we have
\begin{equs}
M_1^2\bigl(\overline{\bphi}(t)\bigr) =K^2(t)G^2(\overline{\phi}_k(t)) =K^2(t)\bigr(  \langle G^2\rangle + \psi'(\overline \phi_k(t))\bigr)~,
\end{equs}
where
\begin{equ}
  \langle G^2\rangle= \frac{1}{2\pi }\int_0^{2\pi} G^2(\phi)\,\d \phi>0~,
\end{equ}
and $\psi$ is a primitive of $G^2 -   \langle G^2\rangle$ on $\torus$ (the fact that $\langle G^2 \rangle$ is strictly
positive follows from \aref{d:23}, which guarantees that $U_{k-1}$, and hence also $G$, are non-constant).

Thus, recalling that $\dt  {\overline \phi_k}(t) = \overline I_k(t) =  \overline I_k(0) $, we obtain after integrating by parts that
\begin{equs}
 \int_0^1 K^2(t) \psi'(\overline \phi_{k}(t)) \,\d t & = \left[\frac {\psi(\overline \phi_k(t))}{\overline I_k(0)}K^2(t) \right]_0^1 - 2\int_0^1\frac {\psi(\overline \phi_k(t))}{\overline I_k(0)}K(t)\dot K(t)\d t.
\end{equs}
Since both terms above are $\OOs(L^{-1})$, we have
\begin{equa}
  \int_0^1 M_1^2\bigl(\overline {\bphi}(t)\bigr)\,\d t &=
\langle G^2\rangle\int_0^1 K^2(t) \,\d t +\OOs(L^{-1})~.
\end{equa}

Observe now that $\inf_{\overline x(0) \in \BB_{L, \rho^*}}\int_0^1 K^2(t) \,\d t$ is independent of $L$, since
$K(t)$ does not depend on $I_k$.
Thus, since the $\OOs(L^{-1})$ above can be made arbitrarily small, the proposition is proved if we can show that
\begin{equ}
	\inf_{\overline x(0) \in \BB_{0, \rho^*}}\int_0^1 K^2(t) \,\d t > 0~.
\end{equ}
By compactness, it suffices to prove that for all $\overline x(0) \in \overline{\BB}_{0, \rho^*}$,
\begin{equs}[eq:K2geq0]
\int_0^1 K^2(t)\,\d t  > 0~.
\end{equs}
Assume by contradiction that $\overline x(0) \in  \overline{\BB}_{0, \rho^*}$
and that $K(t) \equiv 0$ on the interval $[0,1]$. By analyticity
(\lref{lem:solutionanalytique}), it is easy to realize that this
implies that for some $\ell < k$,  $U_\ell ''(\overline\phi_\ell(t) -
\overline\phi_{\ell-1}(t)) \equiv 0$ for all $t > 0$ (not only in
$[0,1]$). We now show that this leads to a contradiction.

Since  $U_\ell ''$ has no flat part by \aref{d:22}, this implies that actually
$\overline\phi_\ell(t) - \overline\phi_{\ell-1}(t)$ is constant. By
Assumption~\ref{d:23}, we have $U_\ell '(\overline\phi_\ell(t) -
\overline\phi_{\ell-1}(t)) \equiv \kappa \neq 0$. But then, the subsystem
$(\overline\phi_{\ell}, \overline I_\ell, \dots, \overline\phi_{k-1}, \overline I_{k-1})$ receives a
constant, non-zero force $\kappa$. Thus, the total momentum of this
subsystem will eventually be arbitrarily large, and so will its energy.
This is a contradiction, since the energy of the system
$(\overline\phi_{1}, \overline I_1, \dots, \overline\phi_{k-1}, \overline I_{k-1})$ is
non-increasing. The proof is complete.
\end{proof}

The next proposition shows that these estimates extend to the dynamics of $\tilde x$, and
also to any subinterval of length 1 in $[0,T]$.
\begin{proposition}\label{prop:intM1sqrealdyn}
There is a constant $\tilde{c}_*>0$ such that for all large enough $L$,
and all $0 \leq t_0 \leq T-1$,
\begin{equs}[eq:intm2gqc]
\int_{t_0}^{t_0+1} M_1^2\bigl(\tilde \bphi(t)\bigr)\,\d t  \geq  \tilde c_*~.
\end{equs}
\end{proposition}
\begin{proof}
We compare the trajectory
$\tilde x(t_0+t)$ and the solution $\overline x(t)$ of \eref{eq:ODEdec1}
with initial condition $\overline x(0) = \tilde x(t_0)$, which belongs
to $\BB_{L, \rho^*}$.
By comparing \eref{eq:dynamicstildeItildephi} and \eref{eq:ODEdec1}, we see that
\begin{equ}
  \sup_{t\in[0,1]} \bigl|
    \tilde x(t_0+t) -\overline{x} (t)\bigr| = \OOs(L^{-2}) ~.
\end{equ}
Using \pref{prop:M1t01}, we then obtain the desired result with $\tilde{c}_*=
\overline{c}_*/2$ and $L$ sufficiently large,
since $M_1$ is uniformly continuous.
\end{proof}

\begin{proof}[Proof of \pref{p:p1not0}]
By \lref{lem:realboundP1},
\begin{equa}[eq:lbM1sq]
  \int_0^T P_1^2(\tilde{x}(t))\,\d t&\ge \frac{1}{L^{4k-6}}
\int_0^TM_1^2\bigl(\tilde{\bphi}(t)\bigr)\,\d t
+\OOs(T/L^{4k-5})\\
&=
\frac{1}{L^{4k-6}}
\int_0^TM_1^2\bigl(\tilde{\bphi}(t)\bigr)\,\d t
+\OOs(L^{2-2k})~.
\end{equa}
Next, we consider the decomposition
  \begin{equ}
    \int_0^T M_1^2(\tilde{\bphi}(t)) \d t \geq \sum_{n=0}^{\lfloor T-1 \rfloor }\int_n^{n+1} M_1^2(\tilde{\bphi}(t)) \d t~.
  \end{equ}
By \pref{prop:intM1sqrealdyn}, we find that for all large enough $L$,
  \begin{equ}
    \int_0^T M_1^2(\tilde{\bphi}(t)) \d t \geq \tilde{c}_*\lfloor T \rfloor ~.
  \end{equ}
Thus, by \eref{eq:lbM1sq}, we indeed get \eref{eq:integralep1sq} if $c = \tilde{c}_*/2$ and $L$ is large enough. This completes the proof.
\end{proof}

\subsection{Neglecting the mixed term}\label{sec:partii}

In this subsection we complete the proof of \tref{t:main}.  Note that \eref{e:dissipative} and \pref{p:p1not0} imply
that
\begin{equ} \label{eq:sect9}
H(x(T)) - H(x(0)) \le -\gamma \int_0^T
\tilde{I}_1^2(t)\, \d t  - \frac{\gamma \alpha c}{L^{2k -3}}
- 2\gamma \int_0^T \tilde{I}_1(t) P_1(\tilde{x}(t))\,\d t\ .
\end{equ}

It remains to prove that the last term in the right-hand side is smaller (in magnitude) than the first two,
from which \tref{t:main} will follow.

First, by \eref{eq:decompositionP1}, we can write
\begin{equs}
P_1 = P_{10} + P_{11}~,
\end{equs}
where
\begin{equs}
	P_{10}&= -\partial_{ \phi_1}\chi\x{k-2} = \OOd_{\NR}^\fin(L^{3-2k})~,\\
P_{11} &= \OOd(L^{5-4k})~.
\end{equs}

By Young's inequality, for any $K>0$,
\begin{equs}
|2 \tilde I_1   P_{11}(\tilde x)|&	 \leq \frac {\tilde I_1^2}{K} + K P^2_{11}(\tilde x)~.
\end{equs}
Thus, we find
\begin{equs}\label{e:newstar}
-2\gamma \int_0^T \tilde  I_1 (t)\,   P_{11}(\tilde x(t))\,\d t \leq \frac 1K\int_0^T { {\tilde I_1^2 }}(t) \d t  	+ \OOs(L^{7-6k})~.
\end{equs}
If we choose first $K$ large enough, and then $L$ large enough, the above is indeed smaller than the two ``good'' terms in \eref{eq:sect9}.

It remains therefore only to deal with
 \begin{equs}[eq:intP100T]
\int_0^T \tilde  I_1 (t)\,   P_{10}(\tilde x(t))\,\d t~.
\end{equs}

The main result of this subsection is the following proposition:
\begin{proposition}\label{prop:borneunnterme}
We have
\begin{equs}[eq:IntP10]
\int_0^T \tilde  I_1 (t)\,   P_{10}(\tilde x(t))\,\d t = \OOs(L^{2-2k})~.
\end{equs}
\end{proposition}

This proposition, together with \eref{e:newstar}, will guarantee that the third term in \eref{eq:sect9} is small in comparison with the first two.
We will prove \pref{prop:borneunnterme} by successive integrations by parts.

The integrand in \eref{eq:intP100T} has the form $\tilde I_1\, \OOd_{\NR}^\fin(L^{3-2k})$.
Our main induction step deals with slightly more general integrands, where the power of $\tilde I_1$ can be different from one. We say that an admissible function $f$ is of class $\RR(p,m)$ with $p, m\geq 0$ if
\begin{equ}
f ( x) =  I_1^m\, f^* ( x)~,
\end{equ}
where $f^* = \OOd_{\NR}^\fin(L^{-p})$.
Note that we have $f = \OOs(L^{-p})$, but only $f = \OOd(L^{m-p})$.

\begin{lemma}\label{lem:102}
Let $f$ be of class $\RR(p,m)$ with $m\geq 0$ and $p\geq 2k-3$. Then, there exist
finitely many functions $g_1, \dots, g_{N}$, each of class $\RR(m_\ell, p_\ell)$ with $m_\ell\geq 0$ and $p_\ell \geq p+1$, such that
\begin{equs}
   \int_0^T f(\tilde x (t)) \d t & = \sum_{\ell=1}^N \int_0^T g_\ell(\tilde x(t))\d t +  \OOs(L^{2-2k})~.
\end{equs}
\end{lemma}
\begin{proof}
Since $f$ is of class $\RR(p,m)$, we have $f =  I_1^m \, f^*$, where $f^* = \OOd^\fin_{\NR}(L^{-p})$.
We can thus define
\begin{equs}
F(x) = Qf(x) =  I_1^m\, Qf^*(x)~,
\end{equs}
which is of class $\RR(p+1, m)$, and hence $\OOs^\fin_{\NR}(L^{-p-1})$. The key observation is that
\begin{equs}[eq:InablaFf]
\sum_{i=1}^n I_i\partial_{\phi_i} F =  I_1^m \sum_{i=1}^n I_i\partial_{\phi_i} Q f^* =  I_1^m f^* = f~.
\end{equs}
(The above is obtained like \eref{eq:cancelNR}, recalling that $f^*$ is non-resonant.)

Moreover, by \lref{lem:eqItilde},
\begin{equs}
\frac {\d }{\d t}F(\tilde x(t))& = \sum_{i=1}^n \frac{\partial F}{\partial { \phi_i}}(\tilde x(t))\left(\tilde I_i(t)+\OOs^\fin_{\R}(L^{-3}) + \OOs(L^{2-2k})\right)\\
& \quad + \sum_{i=1}^n \frac{\partial F}{\partial { I_i}}(\tilde x(t))\left(\OOs^\fin_{\R}(L^0) + \OOs(L^{3-2k})
 -\gamma \delta_{1,i} \tilde I_1(t)\right)\\
&  = \sum_{i=1}^n \tilde I_i(t)\frac{\partial F}{\partial { \phi_i}}(\tilde x(t)) + \OOs^\fin_{\NR}(L^{-p-1})  + \tilde I_1(t)\OOs^\fin_{\NR}(L^{-p-1}) + \OOs(L^{5-4k})~,
\end{equs}
where we have used that $\partial_{\phi_i} F $ and $\partial_{I_i} F $ are $\OOs^\fin_{\NR}(L^{-p-1})$, that $p\geq 2k-3$, and that the product of a non-resonant term and a resonant one is non-resonant. Using now \eref{eq:InablaFf} and decomposing the two $\OOs^\fin_{\NR}(L^{-p-1})$ above, we find
\begin{equs}
\frac {\d }{\d t}F(\tilde x(t))& = f(\tilde x(t)) - \sum_{\ell=1}^N g_\ell(\tilde x(t)) + \OOs(L^{5-4k})~,
\end{equs}
where $g_1, \dots, g_N$ are as in the statement of the lemma. But then, integrating both sides gives
\begin{equs}
   \int_0^T f(\tilde x (t)) \d t & = F(\tilde x(t))|_{0}^T + \sum_{\ell=1}^N \int_0^T g_\ell(\tilde x(t))\d t +  \OOs(L^{2-2k})~,
\end{equs}
from which the result follows since,  $F = \OOs(L^{-p-1}) =  \OOs(L^{2-2k})$.
\end{proof}

\begin{proof}[Proof of \pref{prop:borneunnterme}]
The integrand in \eref{eq:IntP10} is of class $\RR(2k-3, 1)$. The result then immediately follows from applying \lref{lem:102} recursively. Indeed, at each step, we have finitely many functions $g_\ell$. Moreover, at each step the degree $p_\ell$ of these functions increases at least by one. Thus, after finitely many steps all the $g_\ell$ have order $p_\ell \geq 5-4k$, and hence $\int_0^T g_\ell(\tilde x(t)) \d t = \OOs(T L^{5-4k}) = \OOs(L^{2-2k})$, which completes the proof.
\end{proof}

We can finally prove the main theorem.
\begin{proof}[Proof of \tref{t:main}]
Part (i) was proved in \pref{p:wasGronwall}, and Part (ii) follows immediately from the estimates above. Indeed, putting together \eref{eq:sect9}, \eref{e:newstar} (with $K$ large enough)  and \eref{eq:IntP10} implies that
\begin{equ}
H(x(T)) - H(x(0)) \le - \frac{\gamma\alpha c}{L^{2k -3}} +\OOs(L^{2-2k})~,
\end{equ}
which, by setting $C_1 = c/2$ and using  the definition of $\OOs$, completes the proof of Part (ii).
\end{proof}

\section{Asymptotic equations}\label{sec:asymptoticeqs}

We have found, through computer simulations and the arguments in \sref{sec:illustration},
that under dissipation,
the system quickly approaches a quasi-stationary state where $I_k$ oscillates
with very small amplitude
around some fixed value $\langle I_k \rangle \approx L$, and
where all other $I_j$ oscillate with very small amplitude around zero.
While we have not been able to
prove that this quasi-stationary state exists, it is still instructive
to derive asymptotic equations which seem to describe it.
We assume here that the potentials are $U_i(\phi_{i+1}-\phi_i) = -\cos(\phi_{i+1}-\phi_i)$.

Since the results of the previous sections show that the value of the action variable $I_k(t)$
changes very slowly, we have
\begin{equ}
\dot{\phi}_k = I_ k \approx L\ ,
\end{equ}
to a high degree of accuracy.  Thus, we can, for long times, approximate
\begin{equ}
\phi_k(t) \approx L t \ , \ \  I_k(t) \approx  L\ .
\end{equ}
(For simplicity, we have assumed that $\phi_k(0) = 0$.)

Now consider the evolution of $I_i$ for $i<k$, which is given by
\begin{equa}
  \dot I_{k-1} &= \sin(\phi_{k-2}-\phi_{k-1})- \sin(\phi_{k-1}-\phi_k) ~,\\
  \dot I_{k-2} &= \sin(\phi_{k-3}-\phi_{k-2})- \sin(\phi_{k-2}-\phi_{k-1}) ~,\\
  &\cdots \\
  \dot I_{1} &= \sin(\phi_{2}-\phi_{1}) - \gamma I_1~.
\end{equa}

Our numerical experiments show that the $\phi_i$, $i<k$,
oscillate around some (common) value, which is a minimum
of the interactions potentials. Since this minimum is degenerate
(as the
system is invariant under global rotation), we assume without loss of generality 
that the $\phi_i$, $i<k$ oscillate around zero.
Our numerical experiments show that, in terms of amplitudes of the
oscillations,
\begin{equ}
1 \gg \phi_{k-1} \gg \phi_{k-2} \gg \dots \gg \phi_1\ .
\end{equ}
Thus, the equation of motion for $I_{k-1}$ can be approximated by
\begin{equ}
\dot{I}_{k-1} = \sin(\phi_{k-2} - \phi_{k-1}) - \sin(\phi_{k-1}- \phi_k) \approx
\sin(\phi_k) \approx  \sin(Lt)\ ,
\end{equ}
so that
\begin{equ}
I_{k-1}(t) \approx -\frac{1}{L} \cos(Lt) \ .
\end{equ}
Since
\begin{equ}
\dot{\phi}_{k-1} = I_{k-1}\ ,
\end{equ}
we immediately find
\begin{equ}
\phi_{k-1}(t) \approx -\frac{1}{L^2} \sin(L t)\ .
\end{equ}

Continuing in this fashion, we arrive at an approximation to the
actions and angles in the quasi-stationary state of the form:
\begin{equa}
I_k& \approx L~,\\
  I_{k-1}& \approx - \cos(Lt)/L~,\\
  \phi_{k-1}& \approx- \sin(Lt)/L^2~,\\
  I_{k-2}& \approx  \cos(Lt)/L^3~,\\
  \phi_{k-2}& \approx \sin(Lt)/L^4~,\\
  I_{k-3}& \approx - \cos(Lt)/L^5~,\\
  \phi_{k-3}& \approx- \sin(Lt)/L^6~,\\
& \cdots ~.
\end{equa} 
This scaling predicted by this simple argument corresponds to \eref{eq:scalingImI}
and is corroborated by the results of the numerical
calculations, as shown in \fref{fig:fig0} and \fref{fig:fig2}.
A similar argument applies to the sites $i>k$.

Note also that in fact
\emph{any} periodic function with non-degenerate quadratic extrema can replace
the cosine above.

We conjecture that a whole family of such quasi-stationary states exists, and that after a ``fast'' initial transient the subsequent dissipative evolution occurs by a slow
motion along this family of quasi-stationary solutions. 
We expect the family of quasi-stationary states to be parameterized by $\langle I_k\rangle $ and the equilibrium position around which the $\phi_i$, $i\neq k$ oscillate.

\section{Decay rate for degenerate potentials}\Label{s:degenerate}\label{sec:nondeg}

In \tref{t:main} we proved a bound on the rate of dissipation of energy for certain trajectories that indicates that
the energy loss per unit time should scale like
$- L^{6-4k}$ when the $k$'th rotator  ($k\geq 2$) initially has very large energy and all other rotators
have initial energy of order $ 1$.  Our numerics indicate that for a system with cosine nearest
neighbor potentials, this decay rate is sharp.
However, we expect that if the potentials $U_i$ violate  \aref{d:23}, the
dissipation rate can be much slower.  Note, in particular, that this happens
 when the potentials have degenerate extrema, {\it i.e.},
$U_i''(\phi)=0$ and $U_i'(\phi)=0$ for some $\phi$ and $i$.

To illustrate the reasons for our expectations (in a non-rigorous way),
consider  a chain of 3
rotators, in which the first and the second are connected by a
degenerate potential, $U(\phi) = {(\cos(\phi)-1)^2}/{2}$, while the second
and third rotators are still coupled by a cosine potential. The Hamiltonian is then
\begin{equs}
H = \sum_{i=1}^3 \frac{I_i^2}2 - \cos(\phi_3-\phi_2) + \frac{(\cos(\phi_2-\phi_1)-1)^2}{2}	~.
\end{equs}
The point here is that the minimum of
$U(\phi) = {(\cos(\phi)-1)^2}/{2}$ is attained at $\phi=0$,
and that the first 3 derivatives at 0 vanish, leaving us with
the expansion $U(\phi) = \frac 1 8 \phi^4 + \OOs(\phi^6) $.

We now proceed to find approximate solutions of the equations of motion following the method sketched
in \sref{sec:asymptoticeqs}, making the same technical assumptions.  We find
\begin{equ}
I_3(t)  \approx L\ ,\ ~{\text{and}}~\ \phi_3(t) \approx Lt\ .
\end{equ}
Likewise, we have
\begin{equ}
I_2(t) \approx -\frac{1}{L} \cos(Lt)\ , ~{\text{and}}~\ \phi_2(t)  \approx -\frac{1}{L^2} \sin(Lt) \  .
\end{equ}

Now, consider the equation for $I_1(t)$. Unlike above, we have
\begin{equs}
\dot{I}_1 &= [1-\cos(\phi_2-\phi_1)]\sin(\phi_2-\phi_1) -\gamma I_1\\
& \approx [1-\cos(\phi_2) ] \sin(\phi_2) \approx  \frac{\phi_2^3}{2} \ ,
\end{equs}
where the second line appealed to our numerical observation that $\phi_2$ is much larger than
$\phi_1$ and $I_1$.
If we then insert our approximation $\phi_2(t) \approx -\frac{1}{L^2} \sin(Lt)$, we find
\begin{equ}
\dot I_1(t) \approx - \frac {\sin^3(Lt)}{2L^6}\ ,
\end{equ}
from which we obtain
\begin{equ}[eq:scalingI1deg]
I_1 \approx \frac {1} {2L^7}\left(\cos(Lt)- \frac 13 \cos^3(Lt)\right)\ .
\end{equ}

Thus, we find that in contrast to what we found under \aref{d:23}, the action variables now
have the following ratios:
\begin{equs}[eq:ratiosdegen]
\frac {|I_2|}{|I_3|} \sim \frac{1}{L^2}~, \qquad 	\qquad \frac {|I_1|}{|I_2|} \sim \frac{1}{L^6}~.
\end{equs}

By \eref{eq:scalingI1deg}, we expect the dissipation rate to scale
like $I_1^2 \sim L^{-14}$, which is much smaller than the $L^{6-4k} = L^{-6}$
that our main theorem gives under  \aref{d:23}.

We checked the statements above numerically for $L=10$, $20$  with
$\gamma=0.1$. In \fref{fig:counter1} we see the ratios in
\eref{eq:ratiosdegen} to a very good approximation (observing that the
bracket in \eref{eq:scalingI1deg} has an amplitude of $2/3$). We have
also checked that, in this case, the dissipation rate indeed scales like $L^{-14}$ if we start from an initial condition that is in the quasi-stationary state.

Note finally that one can make the situation much worse: by choosing
a potential $(1-\cos(\phi_2-\phi_1))^{r}$, we would obtain $I_1 \sim  L^{1-4r}$.

\begin{figure}
\begin{center}  \includegraphics[width=\textwidth, trim={0.3cm, 1cm, 0.35cm, 0}]{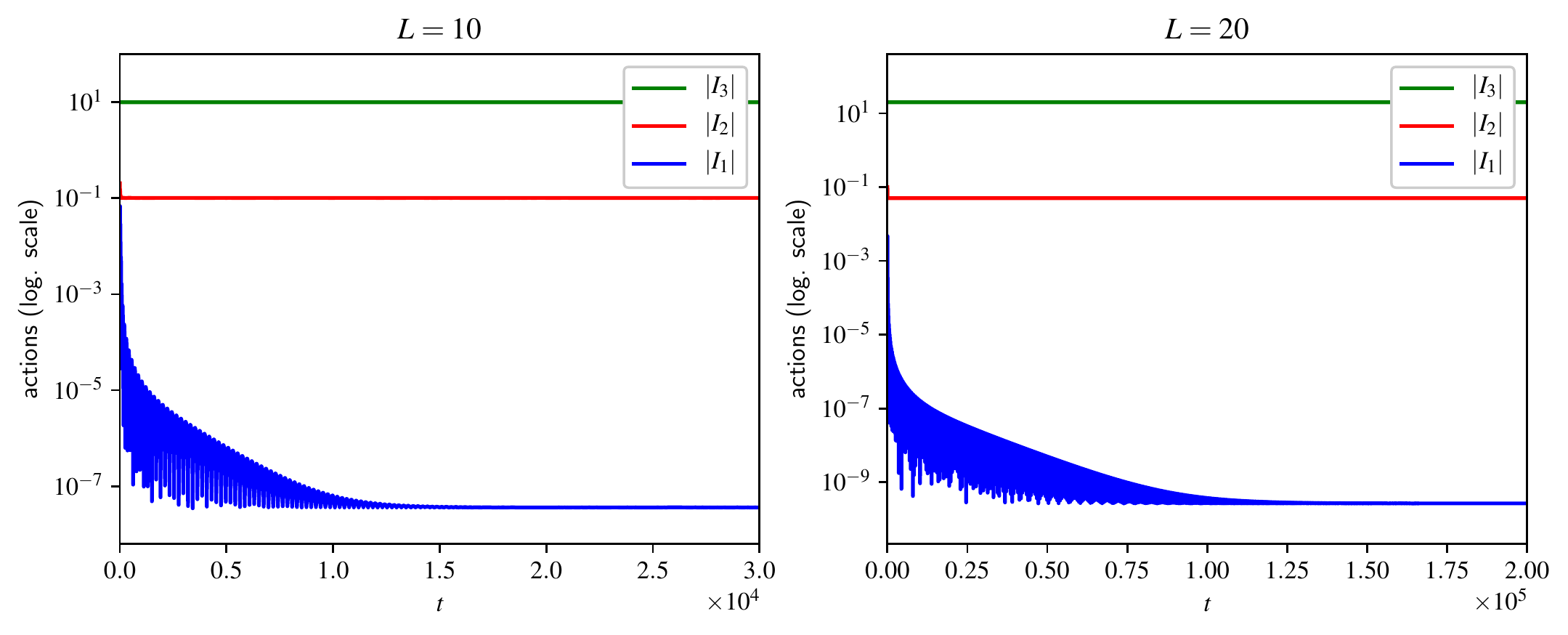}
\end{center}
  \caption{Maximum of $|I_i|$ over intervals of length $2\pi/L$ when the potential $U_1$ is degenerate, for $n=k=3$ and $\gamma=0.1$, with $L=10$ (left) and $L=20$ (right). Note that now the ratios in the quasi-stationary state are $|I_1/I_2|\sim
    L^{-6}$ and $|I_2/I_3|\sim L^{-2}$.}\label{fig:counter1}
\end{figure}

\begin{appendices}

\section{Bounds on analytic functions}\label{sec:thebounds}

The most convenient way to bound the actions of $e^{\chi}$ on a function
$f$ is provided by the analytic methods used by \cite{poeschel_1990}. We adapt
them here to our needs, where, instead of the usual small parameter,
we have here large momentum $I_k$.

We work in the domains $\DD_{L,r,\sigma}$ defined in \dref{def:drs}, and
use the norm (defined in \eref{e:normrs}):
\begin{equa}
 \norm{f}{L,\rs}=\sup_{(\bI,\bphi)\in \DD_{L,\rs}} |f(\bI,\bphi)|~.
\end{equa}

\begin{remark}\label{rem:sigma}
  We will tacitly assume throughout that $\sigma\le1$.
\end{remark}

We also assume that the functions $f, f_1, f_2, \dots$
have the form \eref{eq:fgeneric}.
The next lemma bounds restrictions of Fourier series.

\begin{lemma}\label{lem:g10}
  If $f$ is
  analytic on $\DD_{L,r,\sigma}$, then for any subset
  $\tilde\NN\subset\NN$, the function
$\tilde f =\sum_{\bmm\in\tilde\NN} f_\bmm(\bI)\, e^{\i  \bmm\cdot
    \bphi}$ is analytic on $\DD_{L,r,\tilde \sigma}$ for any
  $0<\tilde\sigma< \sigma$ and
  \begin{equ}
    \norm{\tilde f}{L,r,\tilde \sigma} \le
    \left(\frac{4}{\sigma-\tilde\sigma}\right)^n \norm{f}{L,\rs}~.
  \end{equ}
\end{lemma}

\begin{proof}
Since $f$ is analytic (in $\bphi$),  we have $|f_\bmm(\bI)| \le \norm{f}{L,\rs} \, e^{-\sigma |\bmm|}$,
and therefore
\begin{equa}
   \norm{\tilde f}{L,r,\tilde\sigma } \le\sum_{\bmm\in\tilde\NN}
   \norm{f}{L,\rs} e^{-(\sigma-\sigma') |\bmm|}~.
\end{equa}
Note now that with $\delta = \sigma-\sigma'>0$, and since by
\rref{rem:sigma}, $\delta \le1$,
one has
\begin{equa}[eq:sumntilderestr]
  \sum_{\bmm\in\tilde\NN}
    e^{-\delta |\bmm|}\le \prod _{i=1}^n \left(\sum_{\mu_i=-\infty}^\infty
    e^{-\delta |\mu_i|}
\right)\le \left(\frac{2}{1-e^{-\delta }} \right)^{n}\le (4/\delta)^{n}~.
\end{equa}

\end{proof}

\begin{proposition}\label{prop:chibound}
  Suppose $f$ is analytic on $\DD_{L,\rs}$ and that the decomposition \eref{eq:fgeneric} contains
finitely many Fourier modes.
If $r<{1}/{(2n|\NN|)}$ (recall \dref{d:norms}), the function $Qf$ defined in \eref{eq:defQ}
is analytic on $\DD_{L,r,\sigma'}$ for any $0<\sigma'<\sigma$, and
\begin{equa}\Label{e:chibound}
  \norm{Qf}{L,r,\sigma' } \le 2 \left(\frac{4}{\sigma - \sigma'}\right) ^{n} \frac{\norm{f}{L,\rs}}{L}~.
\end{equa}
\end{proposition}

\begin{proof}
By \lref{lem:lhalf}, each summand in the definition of $Qf$ is analytic on $\DD_{L, r, \sigma'}$.
Since the sum is finite, $Qf$ is analytic on the same domain.
Since $|f_\bmm(\bI)| \le \norm{f}{L,\rs} \, e^{-\sigma |\bmm|}$, and by \lref{lem:lhalf},
we find
\begin{equa}
   \norm{Qf}{L,r,\sigma' } \le\sum_{\bmm\in\NN^{\NR}}
   \frac{\norm{f}{L,\rs} e^{-(\sigma-\sigma') |\bmm|}}{L/2}~.
\end{equa}
The assertion follows using again \eref{eq:sumntilderestr}.
\end{proof}

\begin{lemma}\label{lem:poissonf1f2}
  Suppose $f_1$ and $f_2$ are analytic on $\DD_{L,r_1, \sigma_1}$ and $\DD_{L,r_2, \sigma_2}$ respectively.
Then $\poiss{f_1}{f_2}$ is analytic on $\DD_{\min(r_1, r_2), \min(\sigma_1, \sigma_2)}$.
Moreover, if $r' < \min(r_1, r_2)$ and $\sigma' < \min(\sigma_1, \sigma_2)$, then one has the bound
  \begin{equs}
    \norm{\poiss{f_1}{f_2}}{L,r',\sigma' }&\le
\left(\frac{n}{L(r_1-r')(\sigma_2 -\sigma')} + \frac{n}{L(r_2-r')(\sigma_1 -\sigma')} \right)
\times \norm{f_1}{L,r_1, \sigma_1} \norm{f_2}{L,r_2, \sigma_2}~.
  \end{equs}
\end{lemma}

\begin{proof}
  By Cauchy's theorem,
for $s=1,2$, and $i=1,\dots,n$,
  \begin{equa}[e:star]
    \bnorm{\frac{\partial f_s}{\partial I_i}}{L,r',\sigma_s} &\le \frac{1}{L(r_s-r')}\norm{f_s}{L,r_s, \sigma_s}~,\\
    \bnorm{\frac{\partial f_s}{\partial \phi_i}}{L,r_s,\sigma' } &\le
    \frac{1}{\sigma_s-\sigma' }\norm{f_s}{L,r_s, \sigma_s}~.
  \end{equa}
We have
\begin{equa}
  \poiss{f_1}{f_2} =\sum_{i=1}^n \frac{\partial f_1}{\partial \phi_i}
\frac{\partial f_2}{\partial I_i}-
\frac{\partial f_1}{\partial I_i}
\frac{\partial f_2}{\partial \phi_i}~.
\end{equa}
Since $f_1$ and $f_2$ are analytic on $\DD_{L,\rs}$, so are their derivatives
and $\poiss{f_1}{f_2}$ is a finite sum of analytic functions and hence
analytic on $\DD_{\min(r_1, r_2), \min(\sigma_1, \sigma_2)}$.
The bound on the norm comes from applying \eref{e:star} to
the $2n$ terms in the sum.
\end{proof}

We have the immediate
\begin{corollary}\label{cor:poissonfg}
	Let $f_1$ and $f_2$ be analytic on $\DD_{L,r, \sigma}$. Then,
for all $0<\sigma' < \sigma$ and $0<r' <r$, we have
  \begin{equs}
    \norm{\poiss{f_1}{f_2}}{L,r',\sigma' }&\le
\frac{2n}{L(r-r')(\sigma -\sigma')}  \norm{f_1}{L,r, \sigma} \norm{f_2}{L,r, \sigma}~.
  \end{equs}
\end{corollary}

To sum the Lie series, we follow P\"oschel \cite{Poeschel_1993}.
We begin by bounding $\ad^\ell$:
\begin{lemma}\label{lem:adl}
Assume $g$ and $f$ are analytic on $\DD_{L,\rs}$. Then for all $\ell\geq 1$,
  \begin{equs}\label{e:adl}
\norm{\ad_g^{\ell} f}{L,r', \sigma'} \leq \left(\frac{4n\ell}{L(r-r')(\sigma-\sigma')}\right)^\ell \norm{g}{L,\rs}^\ell\norm{f}{L,r, \sigma}~.
\end{equs}
\end{lemma}
\begin{proof}
We fix $\ell\geq 1$ and estimate $ \ad_g^\ell f $ using
a sequence of nested domains.
For $s =0, 1, 2 \dots, \ell$, we write
\begin{equs}
	z_s  \equiv  (r_s , \sigma_s ) \equiv \left(r' + \frac{s (r-r')}{2\ell},  \sigma' + \frac{s (\sigma-\sigma')}{2\ell}\right)~.
\end{equs}
Applying \lref{lem:poissonf1f2} we find for $s =0, \dots, \ell-1$,
\begin{equs}
\norm{\ad_g^{\ell-s } f}{L,z_{s }} &	\leq \left(\frac{n}{L(r-r_{s })(\sigma_{s +1} -\sigma_{s })} + \frac{n}{L(r_{s +1} - r_{s })(\sigma -\sigma_{s })} \right)\\
& \qquad \times \norm{g}{L,\rs}\norm{\ad_g^{s-\ell -1}f}{L,z_{s +1}} \\
& \leq \frac{4n\ell}{L(r-r')(\sigma-\sigma')} \norm{g}{L,\rs}\norm{\ad_g^{\ell-s -1}f}{L,z_{\ell +1}}~,
\end{equs}
where we have used that $(r-r_s ) \geq (r-r')/2$, and $(\sigma-\sigma_s ) \geq (\sigma-\sigma')/2$, while
$r_{s +1} - r_{s } = (r-r')/\ell$ and $\sigma_{s +1} - \sigma_{s } = (\sigma-\sigma')/\ell$. Iterating this
we get
\begin{equs}
\norm{\ad_g^{\ell} f}{L,r', \sigma'} &= \norm{\ad_g^{\ell} f}{L,z_{0}}\leq \left(\frac{4n\ell}{L(r-r')(\sigma-\sigma')}\right)^\ell \norm{g}{L,\rs}^\ell\norm{f}{L,r, \sigma}~,
\end{equs}
where we have also used that $ \norm{f}{L,z_\ell}\leq \norm{f}{L,r, \sigma}$.
\end{proof}

\begin{proposition}\label{prop:67}
Assume that $\|g\|_{L,r, \sigma} < \infty$ and $\|f\|_{L,r, \sigma} <
\infty$ for some $r, \sigma >0$ and all large enough $L$. Moreover, let $(a_\ell)_{\ell \geq 0}$ be a bounded sequence.
Then,
for all $0<\sigma' < \sigma$, $0<r' <r$, and sufficiently large $L$ we have
for $0 \leq \ell_0 \leq \ell_1 \leq \infty$,
\begin{equa}\Label{e:echi}
 \|\sum_{\ell=\ell_0}^{\ell_1} \frac{a_\ell}{\ell!} \ad_{g}^\ell f \|_{L,r', \sigma'} \leq
2\left(\frac{4ne}{(\sigma-\sigma')
(r-r')}\frac{\norm{g}{L,\rs}}{L}\right)^{\ell_0}\norm{f}{L,\rs} \cdot  \sup_{\ell \geq 0}|\alpha_\ell|~.
\end{equa}

\end{proposition}
\begin{proof}
Without loss of generality we assume that $\sup_{\ell \geq 0}|\alpha_\ell| = 1$.
Using \eref{e:adl} and the bound $ \ell^\ell  / \ell!\le e^\ell $, we get
\begin{equa}
 \|\sum_{\ell=\ell_0}^{\ell_1} \frac{a_\ell }{\ell!} \ad_{g}^\ell f \|_{L,r', \sigma'}  &\le
\sum_{\ell=\ell_0}^{\infty } \frac{(4n\ell)^\ell}{\ell!(L(\sigma-\sigma')
  (r-r'))^\ell}\norm{g}{L,\rs}^\ell\norm{f}{L,\rs}\\
& \le \sum_{\ell=\ell_0}^{\infty } \left(\frac{4ne }{(\sigma-\sigma')
  (r-r')}\frac{\norm{g}{L,\rs}}{L}\right)^\ell\norm{f}{L,\rs}\\
& = \left(\frac{4ne}{(\sigma-\sigma')
(r-r')}\frac{\norm{g}{L,\rs}}{L}\right)^{\ell_0}\norm{f}{L,\rs} \times
\\
& \qquad \sum_{\ell=0}^\infty \left(\frac{4ne }{(\sigma-\sigma')
  (r-r')}\frac{\norm{g}{L,\rs}}{L}\right)^{\ell}~,\Label{e:last}
\end{equa}
which gives the desired result provided $L$ is sufficiently
large so that the series in the last line is bounded by 2.
Note that the bound is decreasing when $L$ increases.

In case $\ell_1 = \infty$,
the result is indeed an analytic function, as the series converges
 uniformly on the open complex domain $\DD_{L,r', \sigma'}$.
\end{proof}

\end{appendices}

\section*{References}

\bibliographystyle{JPE}
\bibliography{refs}

\ack This work was supported by an ERC Advanced grant  290843:BRIDGES (JPE, NC) and by
the Swiss National Science Foundation Grant
165057 (NC).
The
research of CEW was supported in part by the US NSF through grant
DMS-1311553.
\end{document}